\documentclass[final,8pt, 1p,times]{elsarticle}
\usepackage{tikz}
\usepackage{amsfonts, amssymb, amsmath}
\usetikzlibrary{shapes, calc, intersections, arrows}
\usepackage{enumerate}
\usepackage{subcaption}
\usepackage{url}
\newcommand{\N}{\mathbb{N}}

\newcommand{\R}{\mathbb{R}}

\newcommand{\B}{\mathbb{B}}
\newcommand{\Sphere}{\mathbb{S}}

\newcommand{\norm}[1]{\left\Vert #1 \right\Vert}
\newcommand{\abs}[1]{\left| #1 \right|}
\newcommand{\conv}{\operatorname{conv}}

\newcommand{\vol}{\operatorname{vol}}

\newcommand{\interior}{\operatorname{int}}
\newcommand{\bd}{\operatorname{bd}}

\newcommand{\pos}{\operatorname{pos}}

\newcommand{\lin}{\operatorname{lin}}

\tikzset{%
    add/.style args={#1 and #2}{
        to path={%
 ($(\tikztostart)!-#1!(\tikztotarget)$)--($(\tikztotarget)!-#2!(\tikztostart)$)%
  \tikztonodes},add/.default={.2 and .2}}
}

\tikzset{%
  mark coordinate/.style={inner sep=0pt,outer sep=0pt,minimum size=2pt,
    fill=black,circle}%
}

\usepackage{mathptmx}
\usepackage{graphicx}

 \newtheorem{theorem}{Theorem}
\newtheorem{lemma}[theorem]{Lemma}
 \newtheorem{proposition}[theorem]{Proposition}
 \newtheorem{corollary}[theorem]{Corollary}
 \newtheorem{definition}[theorem]{Definition}
\newproof{proof}{Proof}

\usepackage{fancyhdr}
\begin{document}

\begin{frontmatter}

\title{Good Clusterings Have Large Volume}

\author[UCD]{Steffen Borgwardt}
\ead{steffen.borgwardt@ucdenver.edu}

\author[TUM]{Felix Happach}
\ead{felix.happach@tum.de}


\address[UCD]{Department of Mathematical and Statistical Sciences,
  University of Colorado Denver, USA}

\address[TUM]{Department of Mathematics and TUM School of Management, Technische~Universit\"at M\"{u}nchen, Germany}

\begin{abstract}%
The clustering of a data set is one of the core tasks in data analytics. Many clustering algorithms exhibit a strong contrast between a favorable performance in practice and bad theoretical worst-cases. Prime examples are least-squares assignments and the popular $k$-means algorithm. We are interested in this contrast and study it through polyhedral theory. Several popular clustering algorithms can be connected to finding a vertex of the so-called bounded-shape partition polytopes. The vertices correspond to clusterings with extraordinary separation properties, in particular allowing the construction of a separating power diagram, defined by its so-called sites, such that each cluster has its own cell. 

First, we quantitatively measure the space of all sites that allow construction of a separating power diagram for a clustering by the volume of the normal cone at the corresponding vertex. This gives rise to a new quality criterion for clusterings, and explains why good clusterings are also the most likely to be found by some classical algorithms. Second, we characterize the edges of the bounded-shape partition polytopes. Through this, we obtain an explicit description of the normal cones. This allows us to compute measures with respect to the new quality criterion, and even compute ``most stable'' sites, and thereby ``most stable'' power diagrams, for the separation of clusters. The hardness of these computations depends on the number of edges incident to a vertex, which may be exponential. However, the computational effort is rewarded with a wealth of information that can be gained from the results, which we highlight through some proof-of-concept computations.
\end{abstract}%


\end{frontmatter}

%


\section{Introduction}

Informed decision-making based on large data sets is one of the key challenges in Operations Research. We are interested in one of the fundamental tasks in data analytics, the {\em clustering} of a data set into disjoint clusters. Data is often represented as a finite set $X\subset \mathbb{R}^d$ in $d$-dimensional Euclidean space. A clustering $C=(C_1,\dots,C_k)$ then is a partition of the set $X$ into disjoint clusters $C_i\subset X$, such that $\bigcup\limits_{i=1}^k C_i = X$ and $C_i \cap C_j = \emptyset$. 

There is a wealth of literature on clusterings methods. We refer to three surveys by \cite{b-02,jmf-99,xw-05}. For many clustering algorithms, there is a strong contrast between an extremely favorable performance in practice and the lack of provable worst-case guarantees in theory. A prime example is the popular $k$-means algorithm proposed by \cite{l-82,m-67}. 
In practice, it typically terminates in just a few iterations and produces human-interpretable results. In a theoretical worst-case, it may take exponentially many iterations even for two-dimensional data (see \cite{v-11}), and the resulting clustering may not capture the structure of the underlying data. In the present paper, we use methods of polyhedral theory to better understand this stunning discrepancy. 

The studies of polyhedra have been a popular approach for applications in Operations Research. There are many cases where the combinatorial properties of these polyhedra revealed deeper insight into the underlying applications, c.f. \cite{AnderesBorgwardt2016, BriedenGritzmann2012, DeLoeraKim2013, Doignon1997, Groetschel1990, GuralnickPerkinson2006, Kalman2014, Onn1993, Suck1992}.
For an introduction to polyhedral theory, we recommend \cite{NemhauserWolseyBuch, SchrijverBuch, Ziegler1995}. Further, we refer the reader to the book review by \cite{Suck1997} of the classical textbook {\em Lectures on Polytopes} by \cite{Ziegler1995} for an in-depth account of polyhedral theory and its applications.

The so-called {\em assignment polytopes} are closely related to our setting and have been studied well, c.f. \cite{Balinski1974,GillLinusson2009, GottliebRao1990}. It is possible to represent the partition of data points $X = \{x_1,\dots,x_n\}\subset \mathbb{R}^d$ into clusters $C_1,\dots,C_k$ by using decision variables $y_{ij}$ to indicate whether data point $x_j$ is assigned to cluster $C_i$ ($y_{ij}=1$) or not ($y_{ij}=0$). Further, many applications specify lower bounds $s_i^-$ and upper bounds $s_i^+$ on the number of points that may be assigned to cluster $C_i$. This gives rise to a simple set of linear constraints that describes all clusterings:
$$
\begin{array}{lcrclcl}
s_i^- & \leq  & \sum\limits_{j=1}^{n}  y_{ij}    & \leq & s_i^+ & \quad & (i\leq k)\\
                         &      & \sum\limits_{i=1}^k y_{ij} & =    &  1
                         & \quad & (j\leq n)\\
                         &      &          y_{ij}
                         & \geq  & 0                   & \quad & (i \leq k,
                         j \leq n).
\end{array}
$$
The first set of constraints makes sure the prescribed cluster size bounds are respected, the second set of constraints guarantees that each data point is assigned to a cluster. With the relaxed constraints $y_{ij}\geq 0$, we obtain a polytope $P$. The coefficient matrix of the constraints is totally unimodular, and the vector on the right-hand side is integral, so the vertices of this polytope are $0,1$-vectors, i.e. all $y_{ij}$ satisfy $y_{ij}\in\{0,1\}$. Each vertex describes a clustering, and vice versa.

To identify ``good'' clusterings, we study a projection of $P$ that includes information on the locations of the data points $X$ in $\mathbb{R}^d$. This projection, the so-called \textbf{bounded-shape partition polytope}, was first introduced in \cite{Barnes1992,Hwang1998}. (For a formal definition, see Section $2.3$.) 
The vertices of bounded-shape partition polytopes exhibit several favorable properties, c.f. \cite{Barnes1992}, such as being a minimizer of the least-squares functional among all clusterings of the same cluster sizes. In particular, these vertices have strong separation properties. They allow for the construction of a {\em separating power diagram}, a generalized Voronoi diagram in $\mathbb{R}^d$ with one polyhedral cell for each cluster. In machine learning and other parts of Operations Research, this separation property is sometimes called {\em piecewise-linear separability}, c.f. \cite{bm-92}. Power diagrams are defined by a set of {\em sites}, $d$-dimensional vectors (one for each cell) which can be seen as the ``centers'' of the cells. See Figure \ref{fig:powerdiagramclustering} for a small example and Section \ref{subsec-clusterings} for a formal defintion.

\begin{figure}
\begin{center}
\begin{tikzpicture}
        
	\coordinate (A1) at (1,0);%
    \coordinate (A2) at (1,-1);%
    \coordinate (A3) at (2,0);%
    \coordinate (A4) at (0,-1);%
    \coordinate (B1) at (-2,0);%
    \coordinate (B2) at (-2,-2);%
    \coordinate (B3) at (-3,-1);%
    \coordinate (B4) at (-3,1);%
    \coordinate (C1) at (-1,2);%
    \coordinate (C2) at (-1,3);%
    \coordinate (C3) at (0,3);%
    \coordinate (C4) at (1,2);%
	\coordinate (SA) at (0.75,-0.5);
	\coordinate (SB) at (-2.52,-0.735);
	\coordinate (SC) at (-0.5,2.25);
    
    \coordinate (X) at  (intersection of 0,0--13,-9 and 0,-8--1,-22);
    \draw (X) -- (-0.357,-3); 
    \draw (X) -- (3,1.916); 
    \draw (X) -- (-3,2.077); 
    
     \fill (A1) circle (0.1);
     \fill (A2) circle (0.1);
     \fill (A3) circle (0.1);
     \fill (A4) circle (0.1);
     \fill[blue] (B1) circle (0.1);
     \fill[blue] (B2) circle (0.1);
     \fill[blue] (B3) circle (0.1);
     \fill[blue] (B4) circle (0.1);
     \fill[red] (C1) circle (0.1);
     \fill[red] (C2) circle (0.1);
     \fill[red] (C3) circle (0.1);
     \fill[red] (C4) circle (0.1);
     \fill (SA) circle (0.05);
   	 \fill[blue] (SB) circle (0.05);
     \fill[red] (SC) circle (0.05);
\end{tikzpicture}
\caption{A separating power diagram for three clusters. The data points of each cluster lie in the interior of their respective cells. The small dots indicate the respective sites.}
\label{fig:powerdiagramclustering}
\end{center}
\end{figure}
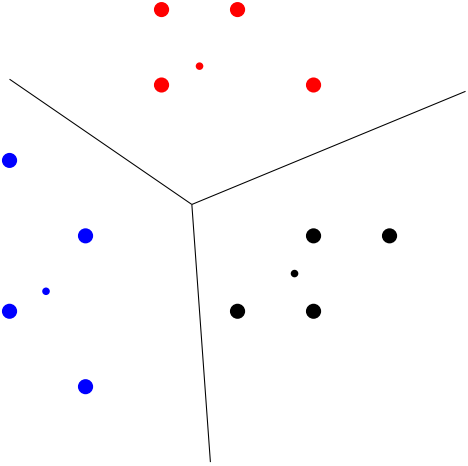

\subsection*{Contributions and Outline}
In Section $2$, we introduce some notation and review related work. We then use the known vertex characterization of \cite{Barnes1992} as a starting point for the new contributions in this paper. We briefly outline these contributions in the following paragraphs. Section $3$ is a complete and self-contained presentation of our results. Section $4$ contains the necessary proofs. We conclude our discussion with some final remarks in Section $5$. Parts of this work are based on the first author's Ph.D. dissertation (\cite{Borgwardt2010}) and the second author's M.Sc. thesis (\cite{Happach2016}).

\noindent{\bf Section 3.1. } First, we devise {\bf a new measure for the quality of a clustering}. We call it the \textbf{volume of a clustering}, since it is a measure for the volume of the normal cone of the vertex encoding this particular clustering in the bounded-shape and single-shape partition polytopes. (The volume of a cone is the standard (finite) Lebesgue measure of the cone intersected with the unit sphere.) This measure is quite different from the stability measures used in the literature, such as least-squares functionals and margins (see \cite{Luxburg2010}). We will exhibit why a large volume indicates a clustering of high quality and distinguish it from the classical measures. This provides an informal {\bf explanation why many clustering algorithms work well in practice}: For example, the computation of a least-squares assignment for fixed cluster sizes is in one-to-one correspondence to linear optimization over a single-shape partition polytope, see \cite{Borgwardt2010}. Further, the $k$-means algorithm can be interpreted as the repeated computation of least-squares assignments with changing sites in each iteration. When choosing random sites, the chance to find a clustering is directly correlated with its volume. This means that clusterings of large volume are found most often, so {\bf the best clusterings are the most likely to be found}.

\noindent{\bf Section 3.2. } Second, we devise an {\bf explicit representation of the normal cones} of vertices of the bounded-shape partition polytope. This is a challenging task, as there is no explicit description of the facets of the polytope. In a generalization of results by \cite{Fukuda2003}, we {\bf characterize the edges of all bounded-shape partition polytopes}. They correspond to so-called {\bf movements or cyclical movements} of items between clusters. This characterization enables us to construct the normal cone of a vertex explicitely and to investigate its structure. By this, we can identify convex areas which contain a representative site for all power diagrams inducing this clustering. We provide some proof-of-concept computations and a running example, in which we compare clusterings of different volumes.

\noindent{\bf Section 3.3. } Finally, we introduce a new \textbf{stability criterion for sites} for a clustering and provide an algorithm for the {\bf computation of optimal sites} in the sense of this stability criterion. These sites are maximally stable with respect to perturbation, i.e. all sites can be perturbed in any direction with a largest possible amount without changing the clustering. We use a classical approach from computational geometry to find such sites: We roll a $p$-norm unit ball into the normal cone and compute where it gets ``stuck''. The center of the ball gives the desired sites. This computation is readily expressed as a mathematical program. Hardness of the computation comes from the fact that there can be exponentially many edges. Of course, this hardness is not surprising -- there are related problems, like the {\em $k$-means problem}, for which the complexity of finding a global optimum (globally optimal sites) is known to be $\mathbb{NP}$-hard, even for $k=2$ (see \cite{adhp-09}) or for data in the Euclidean plane, c.f. \cite{mnv-12}.

\section{Preliminaries}\label{sec-preliminaries}

We begin with some standard notation. Let $A \subseteq \R^d$. Then $\lin(A)$ denotes the minimal linear subspace containing $A$. If $A$ is convex, we call $\interior(A)$ the \textbf{interior} and $\bd(A)$ the \textbf{boundary} of $A$.

\subsection{Clusterings, Least-Squares Assignments and Power Diagrams}\label{subsec-clusterings}

Let throughout this paper $n,d,k \in \N := \{1,2,\dots\}$ be fixed. Let $X := \{x_1,\dots,x_n\} \subseteq \R^d$ be a set of $n$ distinct non-zero data points and for $m \in \N$ define $[m] := \{1,\dots,m\}$. We call a partition $C:=(C_1,\dots,C_k)$ of $X$ a \textbf{clustering} and call $\abs{C}:=(\abs{C_1},\dots,\abs{C_k})$ its \textbf{shape}. For $i \in [k]$, we call $C_i$ the \textbf{$i$-th cluster} of $C$ and $\abs{C_i}$ its \textbf{size}.  Let $s^{-} := (s^{-}_1,\dots,s^{-}_k)$, $s^{+} := (s_1^{+},\dots,s_k^{+}) \in \N^k$ such that $0 \leq s^{-}_i \leq s^{+}_i \leq n$ for all $i \in [k]$ be the lower and upper bounds on the cluster sizes.

A clustering $C$ is said to be \textbf{feasible} if it satisfies $s^{-} \leq \abs{C} \leq s^{+}$ componentwisely. We will only consider feasible clusterings in this paper. $C$ is called \textbf{separable} if all pairs of clusters are linearly separable, i.e.\ for all $i, j \in [k]$, $i \not= j$ there is $a_{ij} \in \R^d$ and $\gamma_{ij} \in \R$ such that $a_{ij}^T x \leq \gamma_{ij} \leq a_{ij}^T y$ for all $x \in C_i$, $y \in C_j$.
The hyperplane separating the clusters is denoted by $H_{(a_{ij},\gamma_{ij})} := \{ x \in \R^d \, | \, a_{ij}^T x = \gamma_{ij} \}$. Analogously, we define $H^{\leq}_{(a_{ij},\gamma_{ij})}$ and $H^{\geq}_{(a_{ij},\gamma_{ij})}$ to be the respective half-spaces.
A \textbf{constrained least-squares assignment (LSA)} for a given set of \textbf{sites} $a_1,\dots,a_k \in \R^d$ is a clustering $C = (C_1,\dots,C_k)$ minimizing 
\begin{equation}\label{LSA}
\sum_{i = 1}^k \sum_{x \in C_i} \norm{x - a_i}_2^2
\end{equation}
over all clusterings with the same shape as $C$. If $C$ minimizes (\ref{LSA}) over all feasible clusterings, we say $C$ is a \textbf{(general) LSA}. We call $a=(a_1^T,\dots,a_k^T)^T \in \R^{d \cdot k}$ the \textbf{site vector}. 

Recall that power diagrams are a generalization of the well-known Voronoi diagrams (see \cite{Aurenhammer1987}). There are several equivalent definitions of power diagrams. We briefly recall the definition that is best for our purposes (see \cite{Borgwardt2015}). Let $a:=(a_1^T,\dots,a_k^T)^T \in \R^{d \cdot k}$ be a site vector with distinct sites $a_1,\dots,a_k \in \R^d$ and let $\alpha_1,\dots,\alpha_k \in \R$. For $i \in [k]$, we call
\begin{equation*}
P_i := \{ x \in \R^d \, | \, (a_j - a_i)^T x \leq \alpha_i - \alpha_j \, \text{ for all } j \in [k] \setminus \{i\} \}
\end{equation*}
the \textbf{$i$-th cell} of the \textbf{power diagram} $(P_1,\dots,P_k)$.

\cite{Aurenhammer1998} showed the following connection between constrained LSAs and power diagrams: If $C$ is a constrained LSA to the site vector $a$, then there is a power diagram with site vector $a$ satisfying $C_i \subseteq \interior(P_i)$ for all $i \in [k]$. On the other hand, if a power diagram $(P_1,\dots,P_k)$ with site vector $a$ satisfies $C_i \subseteq P_i$ for all $i \in [k]$, then $C$ is a constrained LSA to the site vector $a$. If $C_i \subseteq P_i$ for all $i \in [k]$, we say the power diagram \textbf{induces} the clustering and call it a \textbf{separating power diagram}. 

\subsection{Movements Between Clusterings}\label{subsec-movements}

In order to compare two clusterings $C := (C_1,\dots,C_k)$, $C' :=(C'_1,\dots,C'_k)$, we define the \textbf{clustering difference graph (CDG)} to be the labeled directed multigraph $CDG(C,C') :=(V,E)$ with node set $V := [k]$ and edge set $E$ constructed as follows: For each $x_j \in C_i \cap C'_l$ with distinct $i,l \in [k]$, there is an edge $(i,l) \in E$ with label $x_j$.
W.l.o.g.\ we delete isolated nodes in the CDG, since these would correspond to clusters that are identical in $C$ and $C'$. We can derive $C'$ from $C$ by applying operations corresponding to the edges of $CDG(C,C')$.

Let $(i_1,i_2)-(i_2,i_3)-\cdots-(i_{t},i_{t+1})$ be an edge path in $CDG(C,C')$ with labels $x_{j_1},\dots,x_{j_t}$. Applying the \textbf{movement}
\begin{equation*}
M: \quad C_{i_1} \stackrel{x_{j_1}}\longrightarrow C_{i_2} \stackrel{x_{j_2}}\longrightarrow \cdots \stackrel{x_{j_{t}}}\longrightarrow C_{i_{t+1}}
\end{equation*} 
to $C$ means deriving the clustering $\bar{C}=(\bar{C}_1,\dots,\bar{C}_k)$ by setting 
$\bar{C}_{i_l} := (C_{i_l} \setminus \{x_{j_l}\}) \cup \{x_{j_{l-1}}\}$ for all $ l \in \{2,\dots,t\}$,
$\bar{C}_{i_1} := C_{i_1} \setminus \{x_{j_1}\}$,
$\bar{C}_{i_{t+1}} := C_{i_{t+1}} \cup \{x_{j_{t}}\}$ and
$\bar{C}_r := C_r$ for all $r \in [k] \setminus \{i_1,\dots,i_{t+1}\}$.

If $i_{t+1} = i_1$, i.e. in case of a cycle, we speak of a \textbf{cyclical movement}. We then obtain $\bar{C}_{i_1} = \bar{C}_{i_{t+1}} :=(C_{i_1}\setminus \{x_{j_1}\}) \cup \{x_{j_t}\}$ and all cluster sizes remain the same. The \textbf{inverse (cyclical) movement} $M^{-1}$ is defined via the corresponding path (cycle) $(i_{t+1},i_t)-(i_t,i_{t-1})-\cdots-(i_2,i_1)$ in $CDG(C',C)$.
Clearly, one can obtain any clustering from any other clustering by (greedily) decomposing their clustering difference graph into paths and cycles and applying the corresponding (cyclical) movements to $C$. If $\abs{C} = \abs{C'}$, then $CDG(C,C')$ decomposes into cycles, i.e. cyclical movements suffice to transform $C$ into $C'$, c.f. \cite{Borgwardt2010}.

Figures \ref{fig:examplemovement} and \ref{fig:examplecyclicalexchange} depict a clustering of twelve data points in $\R^2$ with $C_1$ black, $C_2$ blue, $C_3$ red, as well as two clusterings which can be derived from it by applying a movement (Figure \ref{fig:examplemovement}) and a cyclical movement (Figure \ref{fig:examplecyclicalexchange}), respectively. 

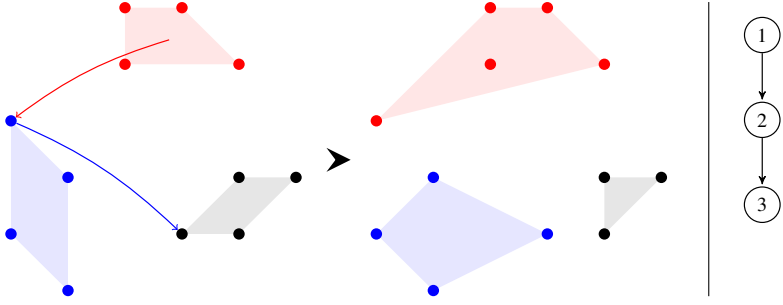
\begin{figure}
\begin{center}
\scalebox{0.75}{
\begin{tikzpicture}
	\coordinate (A1) at (1,0);%
    \coordinate (A2) at (1,-1);%
    \coordinate (A3) at (2,0);%
    \coordinate (A4) at (0,-1);%
    \coordinate (B1) at (-2,0);%
    \coordinate (B2) at (-2,-2);%
    \coordinate (B3) at (-3,-1);%
    \coordinate (B4) at (-3,1);%
    \coordinate (C1) at (-1,2);%
    \coordinate (C2) at (-1,3);%
    \coordinate (C3) at (0,3);%
    \coordinate (C4) at (1,2);%
    
     \fill[black!10] (A1) -- (A3) -- (A2) -- (A4) -- cycle;
     \fill[blue!10] (B1) -- (B2) -- (B3) -- (B4) -- cycle;
     \fill[red!10] (C1) -- (C2) -- (C3) -- (C4) -- cycle;    
    
     \fill (A1) circle (0.1);
     \fill (A2) circle (0.1);
     \fill (A3) circle (0.1);
     \fill (A4) circle (0.1);
     \fill[blue] (B1) circle (0.1);
     \fill[blue] (B2) circle (0.1);
     \fill[blue] (B3) circle (0.1);
     \fill[blue] (B4) circle (0.1);
     \fill[red] (C1) circle (0.1);
     \fill[red] (C2) circle (0.1);
     \fill[red] (C3) circle (0.1);
     \fill[red] (C4) circle (0.1);
     
     \node (1) at (0,2.5) {};
     \path[->, semithick, shorten >=3pt, shorten <= 3pt] (B4) edge[bend left=10, blue] (A4);
     \path[->, semithick, shorten >=3pt, shorten <= 3pt] (1) edge[bend right=10, red] (B4);
\end{tikzpicture}
}
\hspace*{0.3cm}
\put(0,50){\tikz \draw[->,line width=1mm,>=stealth] (0,50)--(4*\unitlength,50);}
\hspace*{0.3cm}
\scalebox{0.75}{
\begin{tikzpicture}
	\coordinate (A1) at (1,0);%
    \coordinate (A2) at (1,-1);%
    \coordinate (A3) at (2,0);%
    \coordinate (A4) at (0,-1);%
    \coordinate (B1) at (-2,0);%
    \coordinate (B2) at (-2,-2);%
    \coordinate (B3) at (-3,-1);%
    \coordinate (B4) at (-3,1);%
    \coordinate (C1) at (-1,2);%
    \coordinate (C2) at (-1,3);%
    \coordinate (C3) at (0,3);%
    \coordinate (C4) at (1,2);%
    
     \fill[black!10] (A1) -- (A3) -- (A2) -- cycle;
     \fill[blue!10] (B1) -- (A4) -- (B2) -- (B3) -- cycle;
     \fill[red!10] (B4) -- (C2) -- (C3) -- (C4) -- cycle; 
    
     \fill (A1) circle (0.1);
     \fill (A2) circle (0.1);
     \fill (A3) circle (0.1);
     \fill[blue] (A4) circle (0.1);
     \fill[blue] (B1) circle (0.1);
     \fill[blue] (B2) circle (0.1);
     \fill[blue] (B3) circle (0.1);
     \fill[red] (B4) circle (0.1);
     \fill[red] (C1) circle (0.1);
     \fill[red] (C2) circle (0.1);
     \fill[red] (C3) circle (0.1);
     \fill[red] (C4) circle (0.1);
\end{tikzpicture}
}
\hspace*{0.3cm}
\vline
\hspace*{0.3cm}
\scalebox{0.75}{
\begin{tikzpicture}[->,>=stealth',shorten >=1pt,auto,node distance=1.5cm,semithick]
\node[circle, draw] (1) {$1$};
\node[circle, draw] (2) [below of=1] {$2$};
\node[circle, draw] (3) [below of=2] {$3$};
\node (4) [below of=3]{};

\path 	(1) edge (2)
		(2) edge (3);
\end{tikzpicture}
}
\caption{Application of a movement, and the corresponding CDG for the two clusterings.}
\label{fig:examplemovement}
\end{center}
\end{figure}

\begin{figure}
\begin{center}
\scalebox{0.75}{
\begin{tikzpicture}
	\coordinate (A1) at (1,0);%
    \coordinate (A2) at (1,-1);%
    \coordinate (A3) at (2,0);%
    \coordinate (A4) at (0,-1);%
    \coordinate (B1) at (-2,0);%
    \coordinate (B2) at (-2,-2);%
    \coordinate (B3) at (-3,-1);%
    \coordinate (B4) at (-3,1);%
    \coordinate (C1) at (-1,2);%
    \coordinate (C2) at (-1,3);%
    \coordinate (C3) at (0,3);%
    \coordinate (C4) at (1,2);%
    
     \fill[black!10] (A1) -- (A3) -- (A2) -- (A4) -- cycle;
     \fill[blue!10] (B1) -- (B2) -- (B3) -- (B4) -- cycle;
     \fill[red!10] (C1) -- (C2) -- (C3) -- (C4) -- cycle;    
    
     \fill (A1) circle (0.1);
     \fill (A2) circle (0.1);
     \fill (A3) circle (0.1);
     \fill (A4) circle (0.1);
     \fill[blue] (B1) circle (0.1);
     \fill[blue] (B2) circle (0.1);
     \fill[blue] (B3) circle (0.1);
     \fill[blue] (B4) circle (0.1);
     \fill[red] (C1) circle (0.1);
     \fill[red] (C2) circle (0.1);
     \fill[red] (C3) circle (0.1);
     \fill[red] (C4) circle (0.1);
     
     \path[->, semithick, shorten >=3pt, shorten <= 3pt] (B4) edge[bend left=10, blue] (A4);
     \path[->, semithick, shorten >=3pt, shorten <= 3pt] (A4) edge[bend left=10, black] (C4);
     \path[->, semithick, shorten >=3pt, shorten <= 3pt] (C4) edge[bend left=10, red] (B4);
\end{tikzpicture}
}
\hspace*{0.3cm}
\put(0,50){\tikz \draw[->,line width=1mm,>=stealth] (0,50)--(4*\unitlength,50);}
\hspace*{0.3cm}
\scalebox{0.75}{
\begin{tikzpicture}
	\coordinate (A1) at (1,0);%
    \coordinate (A2) at (1,-1);%
    \coordinate (A3) at (2,0);%
    \coordinate (A4) at (0,-1);%
    \coordinate (B1) at (-2,0);%
    \coordinate (B2) at (-2,-2);%
    \coordinate (B3) at (-3,-1);%
    \coordinate (B4) at (-3,1);%
    \coordinate (C1) at (-1,2);%
    \coordinate (C2) at (-1,3);%
    \coordinate (C3) at (0,3);%
    \coordinate (C4) at (1,2);%

	\fill[black!10] (A1) -- (A2) -- (A3) -- (C4) -- cycle;
	\fill[blue!10] (B1) -- (A4) -- (B2) -- (B3) -- cycle;
	\fill[red!10] (C1) -- (C3) -- (C2) -- (B4) -- cycle;      
        
     \fill (A1) circle (0.1);
     \fill (A2) circle (0.1);
     \fill (A3) circle (0.1);
     \fill[blue] (A4) circle (0.1);
     \fill[blue] (B1) circle (0.1);
     \fill[blue] (B2) circle (0.1);
     \fill[blue] (B3) circle (0.1);
     \fill[red] (B4) circle (0.1);
     \fill[red] (C1) circle (0.1);
     \fill[red] (C2) circle (0.1);
     \fill[red] (C3) circle (0.1);
     \fill (C4) circle (0.1);

\end{tikzpicture}
}
\hspace*{0.3cm}
\vline
\hspace*{0.3cm}
\scalebox{0.75}{
\begin{tikzpicture}[->,>=stealth',shorten >=1pt,auto,node distance=1.5cm,semithick]
\node[circle, draw] (1) {$1$};
\node[circle, draw] (2) [below of=1] {$2$};
\node[circle, draw] (3) [below of=2] {$3$};
\node (4) [below of=3]{};

\path 	(1) edge (2)
		(2) edge (3)
		(3) edge[bend right] (1);
\end{tikzpicture}
}
\caption{Application of a cyclical movement, and the corresponding CDG for the two clusterings.}
\label{fig:examplecyclicalexchange}
\end{center}
\end{figure}
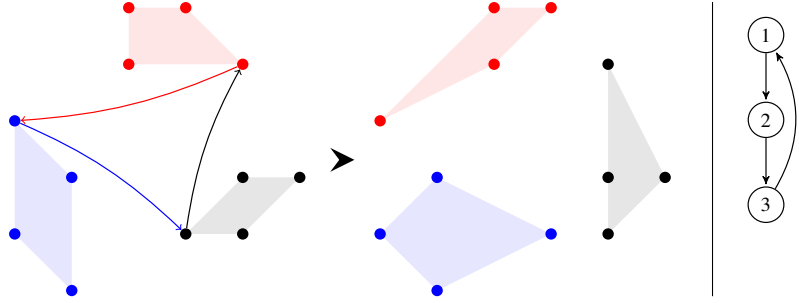

\subsection{Bounded-Shape and Single-Shape Partition Polytopes}\label{subsec-boundedshapepartitionpolytope}

The polytope we are studying was introduced in \cite{Barnes1992} and \cite{Hwang1998}:

For a clustering $C = (C_1,\dots,C_k)$ and $i \in [k]$, let $\sigma_i := \sum\limits_{x \in C_i} x \in \R^d$. The \textbf{clustering vector} of $C$ is $w(C) := (\sigma_1^T,\dots,\sigma_k^T)^T \in \R^{d \cdot k}$. Then $\mathcal{P}^{\pm} (X,k,s^{-},s^{+}) := \conv(\{w(C) \ | \ C $ is feasible$\})$ is called the \textbf{bounded-shape partition polytope}. If $s= s^{-} = s^{+}$, then we call $\mathcal{P}^{=}_s = \mathcal{P}^{\pm}(X,k,s,s)$ the \textbf{single-shape partition polytope}. Another interesting special case is the \textbf{all-shape partition polytope} $\mathcal{P}$ investigated by \cite{Fukuda2003} which is obtained by choosing $s^{-} = (0,\dots,0)$ and $s^{+} = (n,\dots,n)$. When the bounds are clear from the context, we use the simpler notation $\mathcal{P}^{\pm} = \mathcal{P}^{\pm}(X,k,s^-,s^+)$ and $\mathcal{P}^{=} = \mathcal{P}^{=}_s$.

Note that $\mathcal{P}^{\pm}$ is a projection of the generalized assignment polytope investigated in \cite{GottliebRao1990}. We want to stress that, since these polytopes are defined as a convex hull, we do not have explicit information on (facet-defining) valid inequalities. This will play an important role in our  discussion, as we have to study the edge structure of the polytopes (which may have exponential size) in order to construct the normal cones of its vertices. 
We use the notation $N_P(v) := \{ a \in \R^d \ | \ a^T v \geq a^T x \ \forall \, x \in P\}$ for the \textbf{normal cone} of $v$ in a polytope $P$. 

We observe the following connection between single-shape and bounded-shape partition polytopes.

\begin{lemma}\label{union}
$\mathcal{P}^{\pm}(X,k,s^-,s^+) = \conv(\bigcup\limits_{s^- \leq s \leq s^+} \mathcal{P}^{=}_s)$.
\end{lemma}

\proof{Proof.}
Let $U := \bigcup\limits_{s^- \leq s \leq s^+} \mathcal{P}^{=}_s$ and let $w(C) \in \mathcal{P}^{\pm}$ be a clustering vector. Then $w(C) \in \mathcal{P}^{=}_{\abs{C}} \subseteq U$. By definition of $\mathcal{P}^{\pm}$, we thus obtain $\mathcal{P}^{\pm} \subseteq \conv(U)$.
On the other hand, clearly $\mathcal{P}^{=}_s \subseteq \mathcal{P}^{\pm}$ for all $s^- \leq s \leq s^+$ and thus $U \subseteq \mathcal{P}^{\pm}$. Taking the convex hull on both sides yields $\conv(U) \subseteq \conv(\mathcal{P}^{\pm}) = \mathcal{P}^{\pm}$ where the last equality is due to the convexity of $\mathcal{P}^{\pm}$. \hfill \qed
\endproof

Note that Lemma \ref{union} implies that every vertex $w(C)$ of $\mathcal{P}^{\pm}$ is also a vertex of $\mathcal{P}^{=}_{\abs{C}}$ and that $N_{\mathcal{P}^{\pm}}(w(C)) \subseteq N_{\mathcal{P}^{=}}(w(C))$ for every clustering vector $w(C)$.
\cite{Barnes1992} gave a first characterization of the vertices of $\mathcal{P}^{\pm}$.

\begin{proposition}[Barnes, Hoffman, Rothblum 1992] \label{boundedshapevertexcharacterization}
The clustering vector $w(C)$ of a clustering $C$ is a vertex of $\mathcal{P}^{\pm}$ if and only if there are $a := (a^T_1,\dots,a_k^T)^T \in \R^{d \cdot k}$ and $\alpha_1,\dots,\alpha_k \in \R$ satisfying the following statements.

\begin{enumerate}
\item If $\abs{C_i} > s^{-}_i$ for $i \in [k]$, then $\alpha_i \leq 0$.
\item If $\abs{C_i} < s^{+}_i$ for $i \in [k]$, then $\alpha_i \geq 0$.
\item If $x_l \in C_i$ for $l \in [n], i \in [k]$, then for all $j \in [k] \setminus \{i\}$ it holds $(a_j - a_i)^T x_l < \alpha_i - \alpha_j$.
\end{enumerate}
\end{proposition}

A proof is given in \cite{Barnes1992} and, with more technical detail, in \cite{Hwang1998}. We call a clustering corresponding to a vertex of $\mathcal{P}^{\pm}$ a \textbf{vertex clustering}. Condition $3$ states linear separability of the clusters with separation directions $a_{ij} := a_j - a_i \in \R^d$ and right-hand sides $\gamma_{ij} := \alpha_i - \alpha_j$ for all $i,j \in [k]$. This implies the existence of a separating power diagram. This property can alternatively be derived from the observation that computation of an optimal constrained LSA corresponds precisly to linear optimization over the corresponding single-shape partition polytope $\mathcal{P}^{=}$ (see \cite{Borgwardt2010}). 

If the scalars $\alpha_1,\dots,\alpha_k$ additionally satisfy conditions $1$ and $2$, then the resulting separation fulfills some additional properties. For example, if two clusters satisfy $s^- < \abs{C_i}, \abs{C_j} < s^+$, then $C_i$ and $C_j$ are ``$0$-separable'' (see \cite{Aviran2002}), in particular they can be separated by a hyperplane containing the origin. Any vector $a \in \interior(N_{\mathcal{P}^{\pm}}(w(C)))$ can be chosen to construct suitable $\alpha_1,\dots,\alpha_k$ such that the properties of Proposition \ref{boundedshapevertexcharacterization} are satisfied, c.f. \cite{Barnes1992}.

This gives the following corollary (see \cite{Borgwardt2010, Happach2016}).

\begin{corollary}\label{boundedshapepartitionpowerdiagram}
Let $C:=(C_1,\dots,C_k)$ such that $w(C)$ is a vertex of $\mathcal{P}^{\pm}$ and let $a :=(a_1^T,\dots,a_k^T)^T \in N_{\mathcal{P}^{\pm}}(w(C)) \subseteq \R^{d \cdot k}$. Then there is a separating power diagram $(P_1,\dots,P_k)$ with site vector $a \in \R^{d \cdot k}$ such that $C_i \subseteq P_i$ for all $i \in [k]$.
If $a \in \interior(N_{\mathcal{P}^{\pm}}(w(C)))$, then $C_i \subseteq \interior(P_i)$ for all $i \in [k]$. If $a \in \bd(N_{\mathcal{P}^{\pm}}(w(C)) \subseteq \R^{d \cdot k}$, then there is an index $i \in [k]$ such that $C_i \cap \bd(P_i) \not= \emptyset$.
\end{corollary}

The statement also holds for the single-shape partition polytope $\mathcal{P}^=$ by replacing ``$\pm$'' by ``$=$''. Further, the normal cone of a vertex clustering of the single-shape partition polytope encodes exactly all site vectors that allow a separating power diagram (see \cite{Borgwardt2010}). Clearly, if $a$ is in the normal cone, then so is $\lambda a$ for every $\lambda > 0$. Thus $a$ and $\lambda a$ yield the same constrained LSA. This was first proven by \cite{Aurenhammer1998}, in a different notation.

\begin{proposition}[Aurenhammer, Hoffmann, Aronov 1998] \label{scalepowerdiagram}
Let $a \in \R^{d \cdot k}$ be a site vector of a constrained LSA. For all $\lambda > 0$, the site vectors $\lambda a$ yield the same constrained LSA.
\end{proposition}

Before we turn to our main results, we would like to mention two tools that will make our arguments easier. First, recall our assumption that the zero vector is not contained in $X$. This is no restriction, since the overall structure of a data set is not changed when translating the whole set by the same vector. Second, we can interpret any movement (recall Section \ref{subsec-movements}) as a translation of the corresponding clustering vector. Let $C, C'$ such that $CDG(C,C')$ is a single path or cycle corresponding to a (cyclical) movement $M$. Then the difference of the clustering vectors $w(M) := w(C') - w(C)$ is called the \textbf{vector of the movement} $M$. Note that the vector of the inverse movement is given by $w(M^{-1}) = - w(M)$.


\section{Main results}\label{sec-results}

We begin each section with a brief overview.

\subsection{Volume of Clusterings}\label{subsec-largecones}

\noindent{\bf Overview.} 
Maximizing the linear objective vector $a=(a_1^T,\dots,a_k^T)^T \in \mathbb{R}^{d \cdot k}$ over the bounded-shape partition polytope yields a vertex clustering. The sites $a_i \in \mathbb{R}^d$ allow the construction of a separating power diagram for the clusters, c.f. \cite{Barnes1992,Borgwardt2010,Hwang1998}.
Moreover, site vectors are invariant under scaling (Proposition \ref{scalepowerdiagram}). Combining the properties listed in Section \ref{subsec-boundedshapepartitionpolytope} enables us to {\bf quantitatively measure the space of all sites} that allow the construction of a separating power diagram inducing a given clustering {\bf by the volume of its normal cone}. This gives rise to a quality measure that we call the {\bf volume of a clustering}.

\vspace*{0.2cm}

\noindent Instead of considering each site vector $a \in \R^{d \cdot k}$ individually, we consider its equivalence class $[a] := \{ \lambda a \ | \ \lambda > 0\}$ and choose the unit vector $\frac{1}{\norm{a}_2} a$ as a representative. This allows us to introduce a notion of ``distance of sites''. Let $L(\gamma)$ be the length of a curve $\gamma$ and $\Sphere^{d \cdot k} := \{ x \in \R^{d \cdot k} \ | \ \norm{x}_2 = 1\}$ be the Euclidean unit sphere.

\begin{definition}[Distance of Sites] \label{distancepowerdiagram}
Let $a, a' \in \Sphere^{d \cdot k}$ be two site vectors. The \textbf{distance} of the equivalence classes $[a]$ and $[a']$ is defined as the distance of the site vectors on the unit sphere, i.e.
\begin{equation*}
d(a,a') := \inf\{ L(\gamma) \ | \ \gamma: [0;1] \mapsto \R^{d \cdot k}, \, \gamma(0)=a, \, \gamma(1) = a', \, \gamma(t) \in \Sphere^{d \cdot k} \ \forall \, t \in [0;1] \}.
\end{equation*}
\end{definition}

Note that $d: \Sphere^{d \cdot k} \times \Sphere^{d \cdot k} \mapsto \R$ is a metric, takes values between $0$ and $\pi$, and that an infimum always exists, because $L(\gamma) \geq 0$ for all $\gamma: [0;1] \mapsto \R^{d \cdot k}$. 

If a vertex clustering $C$ of $\mathcal{P}^{\pm}$ or $\mathcal{P}^{=}$ has a large normal cone, then a randomly chosen site vector is likely to lie in its cone. By the previous observations, the chosen site vector defines a separating power diagram inducing $C$.
We are interested in measuring the volume of the normal cones of the bounded-shape and single-shape partition polytopes in order to characterize ``good'' clusterings.

We follow the notation of \cite{BonifasEisenbrand2014} who used the volume of normal cones for the studies of combinatorial diameters. 
For a cone $K \subseteq \R^{d \cdot k}$, we call $B(K) := K \cap \Sphere^{d \cdot k}$ the \textbf{base of $K$}. The \textbf{volume of $K$} is then defined as the $(d \cdot k -1)$-dimensional volume of $B(K)$ and is denoted by $\vol(K)$.
A set $A \subseteq \Sphere^{d \cdot k}$ is called \textbf{spherically convex}, if for all $x,y \in A$ the geodesic $\gamma: [0;1] \mapsto \Sphere^{d \cdot k}$ connecting $x$ and $y$ with $\gamma(0)=x$ and $\gamma(1) = y$ is contained in $A$. Recall that a geodesic is a curve on the sphere with shortest length. Note that the base of a cone is spherically convex itself. These notions allow us to formally introduce two variants of a new term which we call the \textbf{volume of a clustering}.

\begin{definition}[Volume of a Clustering]\label{qualitymeasure}
Let $C$ be a feasible clustering and let $N_{\mathcal{P}^{\pm}}(w(C))$ and $N_{\mathcal{P}^{=}}(w(C))$ be its normal cones of $\mathcal{P}^{\pm}$ and $\mathcal{P}^{=}$, respectively. We define
\begin{equation*}
\mu_{\pm}(C) := \frac{\text{vol}(N_{\mathcal{P}^{\pm}}(w(C)))}{\text{vol}(\R^{d \cdot k})}
\end{equation*}
to be the \textbf{BHR volume of $C$} (named after \cite{Barnes1992} who first studied this polytope) and
\begin{equation*}
\mu_{=}(C) := \frac{\text{vol}(N_{\mathcal{P}^{=}}(w(C)))}{\text{vol}(\R^{d \cdot k})}
\end{equation*}
to be the \textbf{LSA volume of $C$}.
\end{definition}

The volume of a clustering puts the volume of the respective normal cones in relation to the volume of the whole space, so by definition $\mu_{\pm}, \mu_{=} \in [0;1]$.
The only difference is that the BHR volume compares the clustering to all feasible clusterings and the LSA volume compares it to all clusterings of the same shape. In the following, we use the simple wording ``volume of a clustering'' to refer to both variants when they behave analogously.

Note that vol$(\R^{d \cdot k})$ equals the area of the surface of $\Sphere^{d \cdot k}$. In practice, permutations of the clusters/sites yield the same clustering and the respective normal cones have equal volume. For our purposes, it is more useful to consider each permutation as an individual clustering. However, if all cluster bounds are symmetric, then the volume defined above only takes values between 0 and $\frac{1}{k!}$, because it only takes one possible permutation into account. By Lemma \ref{union}, $\mu_{\pm} (C) \leq \mu_{=} (C)$ for all clusterings $C$.

The volume of a clustering quantifies the ``fraction'' of all site vectors of separating power diagrams that induce the given clustering. This is quite a different concept than the two widely used quality measures for clusterings: the least-squares functional, which measures the quadratic Euclidean distance of the data points to their respective site, and the so-called {\em margin}, which measures the smallest Euclidean distance of data points to the nearest separating hyperplane (see \cite{Borgwardt2015}). Let us exhibit the difference between the LSA volume and the classical measures. (The BHR volume behaves analogously.) 

Consider Figure \ref{fig:volumebetterthanlsa}, which depicts two similar data sets, where three ``point clouds'' of the same structure are at different distance to each other. The least-squares values of both clusterings are equal, whereas the right-hand clustering has a larger volume. This can be seen by noting that all power diagrams that induce the left-hand clustering also induce the right-hand one. Conversely, there are many power diagrams that induce the right-hand clustering, but not the left-hand one. This means that the set of site vectors that give the left-hand clustering is a strict subset of the set of site vectors that give the right-hand one.  Informally, unlike a least-squares functional, the volume is able to take the distance between different clusters into account.

\begin{figure}
\centering
\includegraphics[scale=0.38]{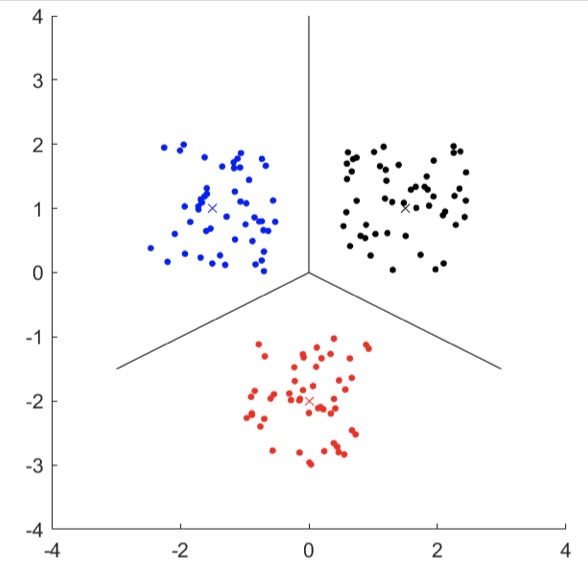}
\hspace*{0.3cm}
\includegraphics[scale=0.38]{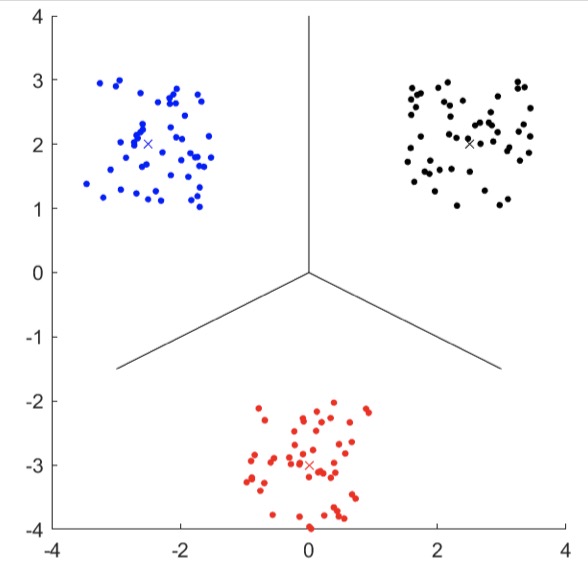}
\caption{Two clusterings with the same least-squares value, but different volume.}
\label{fig:volumebetterthanlsa}
\end{figure}

Now, consider Figure \ref{fig:volumebetterthanmargin}. The only difference between the data sets lies in the vertical scaling of the points in the top clusters and the horizontal scaling of the points in the bottom cluster. Both clusterings have equal margin, but the one on the right has a lower volume. It is easy to see that the separating hyperplanes on the right can be perturbed less due to the long, drawn-out clusters. This implies that the site vectors that define a separating power diagram inducing the right-hand clustering is a strict subset of the site vectors for the left-hand clustering. While the margin informally measures the ability to shift the separating hyperplanes closer to the clusters (without changing their directions), the volume measures the ability to perturb the normals of the separating hyperplanes.


\begin{figure}
\centering
\includegraphics[scale=0.38]{VolumeClustering.png}
\hspace*{0.3cm}
\includegraphics[scale=0.38]{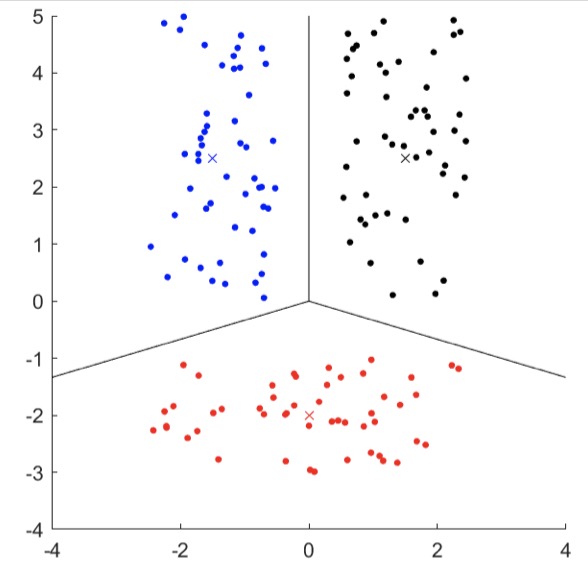}
\caption{Two clusterings with the same margin, but different volume.}
\label{fig:volumebetterthanmargin}
\end{figure}

Note that it does not matter whether we consider $\R^{d \cdot k}$ or the affine hull of the polytope in Definition \ref{qualitymeasure}, because all normal cones have a lineality space $\lin(\mathcal{P}^{\star} - v)^{\perp}$ with $v \in \mathcal{P}^{\star}$ for $\star \in \{ \pm, =\}$, respectively. Since we consider the volume of a cone relative to the whole space, we can restrict to the affine hull of the polytopes (see \cite{Happach2016} for further details).

In fact, one can show that $\mathcal{P}^{\pm}$ and $\mathcal{P}^{=}$ are contained in a $(d \cdot (k-1))$-dimensional affine subspace. Let $a=(\bar{a}^T,\dots,\bar{a}^T)^T \in \R^{d \cdot k}$ with $\bar{a} \in \R^d$ and consider an arbitrary clustering vector $w(C)$. Then
\begin{equation*}
a^T w(C) = \sum_{i =1}^k \bar{a}^T \sigma_i = \sum_{i=1}^k \sum_{x \in C_i} \bar{a}^T x = \sum_{j=1}^n \bar{a}^T x_j,
\end{equation*}
so $\mathcal{P}^{\star} \subseteq \{ x \in \R^{d \cdot k} \ | \ (\bar{a}^T,\dots,\bar{a}^T) x = \sum\limits_{j=1}^n \bar{a}^T x_j \}$ for $\star \in \{ \pm, =\}$. Thus the dimension is at most $d \cdot k - d = d \cdot (k-1)$.


We conclude the section by showing that the volume of a clustering is well-defined, and by providing an explicit algebraic description.

\begin{theorem}\label{measure}
The measures $\mu_{\pm}$ and $\mu_{=}$ are well-defined. For $\star \in \{ \pm , = \}$, $C$ is a vertex clustering of $\mathcal{P}^{\star}$ if and only if $\mu_{\star}(C) > 0$ and then
\begin{equation}\label{integralmeasure}
\mu_{\star}(C) = \frac{\Gamma(\frac{d \cdot k}{2})}{2\pi^{\frac{d \cdot k}{2}}} \int\limits_{\bd(B(N_{\mathcal{P}^{\star}}(w(C))))} d(z,a) \ d a ,
\end{equation}
with $z \in \interior(N_{\mathcal{P}^{\star}}(w(C))) \cap \Sphere^{d \cdot k}$ and $\Gamma$ being the Gamma function.
\end{theorem}

Note that the coefficient in (\ref{integralmeasure}) is the inverse of the area of the surface of $\Sphere^{d \cdot k}$. We postpone the proof of this theorem to Section \ref{sec-proofofmeasure}.

\subsection{Site Vectors in the Normal Cone}\label{subsec-edgesofpolytope}

\noindent{\bf Overview.} We derive an explicit description of the normal cones of the bounded-shape partition polytope via a characterization of the edges of the polytope. As a corollary, we also obtain a description for the normal cones of the single-shape partition polytope. Informally, two clusterings correspond to neighboring vertices of the polytope if they differ by only a {\bf single movement or cyclical movement}. The explicit description of the normal cone allows a representation of the set of sites that define a separating power diagram for a vertex clustering in the form of $k$ convex areas in the space of the underlying data set.

\vspace*{0.2cm}

\noindent It is well-known that the edges incident to a vertex of a polytope are normal vectors to the facets of the normal cone of the vertex. Therefore, in order to obtain an explicit representation of the normal cone, we characterize the edges of the bounded-shape partition polytope. Our characterization generalizes previous results for the special case $d = 1$, c.f. \cite{Gao1999}.

The closest result in the literature is by \cite{Fukuda2003}, who characterized the neighborhood of a vertex $w(C)$ of the all-shape partition polytope $\mathcal{P}$. They showed that edges incident to $w(C)$ correspond to movements of the form $C_i \stackrel{x}\longrightarrow C_j$ or $C_i \stackrel{x}\longrightarrow C_j \stackrel{\nu x}\longrightarrow C_i$ with $\nu < 0$ and $i,j \in [k]$, $i \not= j$. 
Moreover, they proved that, if $\lin(\{x\}) \cap X = \{x\}$ for all $x \in X$ then all edges of $\mathcal{P}$ correspond to movements that move a single element from one cluster to another. The following theorem extends their result to general lower and upper bounds $s^{-}$ and $s^{+}$.

\begin{theorem}\label{boundedshapepartitionedges}
Let $C := (C_1,\dots,C_k)$, $C' := (C'_1,\dots,C'_k)$ be two clusterings such that $w(C)$ and $w(C')$ are adjacent vertices of $\mathcal{P}^{\pm}$. Suppose that no three points in $X$ lie on a single line. Then $C$ and $C'$ differ by a single (cyclical) movement or by two movements and there are distinct $i,j \in [k]$ such that both are of the form $C_i \rightarrow C_j$.
\end{theorem}

The different cases of Theorem \ref{boundedshapepartitionedges} are depicted in Figure \ref{fig:boundedshapepartitionedges}. Note that the case of two movements can only occur if all sites lie on a line. For the single-shape partition polytope, we obtain the following corollary.

\begin{corollary}\label{singleshapepartitionedges}
Let $C := (C_1,\dots,C_k)$, $C' := (C'_1,\dots,C'_k)$ be two clusterings such that $w(C)$ and $w(C')$ are adjacent vertices of $\mathcal{P}^{=}$. If no four points in $X$ lie on a single line, then $C$ and $C'$ only differ by a single cyclical movement.
\end{corollary}

\begin{figure}[h]
\centering
\begin{subfigure}{14cm}
\centering
\begin{tikzpicture}[->,>=stealth',shorten >=1pt,auto,node distance=1.5cm,semithick]
	\node[circle, draw] (1) {$i_1$};
	\node[circle, draw] (2) [right of=1] {$i_2$};
	\node (3) [right of=2] {$\cdots$};
	\node[circle, draw] (4) [right of=3] {$i_t$};
	
	\path 	(1) edge (2)
			(2) edge (3)
			(3) edge (4)
			(4) edge [bend left=20] (1);
\end{tikzpicture}
\caption{\scriptsize The clustering difference graph $CDG(C,C')$ of adjacent $C,C'$ with $\abs{C} = \abs{C'}$.}
\end{subfigure}

\vspace*{0.5cm}

\begin{subfigure}{14cm}
\centering
\begin{tikzpicture}[->,>=stealth',shorten >=1pt,auto,node distance=1.5cm,semithick]
	\node[circle, draw] (1) {$i_1$};
	\node[circle, draw] (2) [right of=1] {$i_2$};
	\node (3) [right of=2] {$\cdots$};
	\node[circle, draw] (4) [right of=3] {$i_t$};
	
	\path 	(1) edge (2)
			(2) edge (3)
			(3) edge (4);
\end{tikzpicture}
\caption{\scriptsize The clustering difference graph $CDG(C,C')$ of adjacent $C,C'$ with $\abs{C} \not= \abs{C'}$.}
\end{subfigure}

\vspace*{0.5cm}

\begin{subfigure}{14cm}
\centering
\begin{tikzpicture}[->,>=stealth',shorten >=1pt,auto,node distance=1.5cm,semithick]
	\node[circle, draw] (1) {$i_1$};
	\node[circle, draw] (2) [right of=1] {$i_2$};
	
	\path 	(1) edge[bend left] (2);
	\path	(1) edge[bend right] (2);

\end{tikzpicture}
\caption{\scriptsize The clustering difference graph $CDG(C,C')$ with two movements.}
\end{subfigure}
\vspace*{0.5cm}
\caption{The possible cases of Theorem \ref{boundedshapepartitionedges} with $i_1,\dots,i_t \in [k]$.}
\label{fig:boundedshapepartitionedges}
\end{figure}
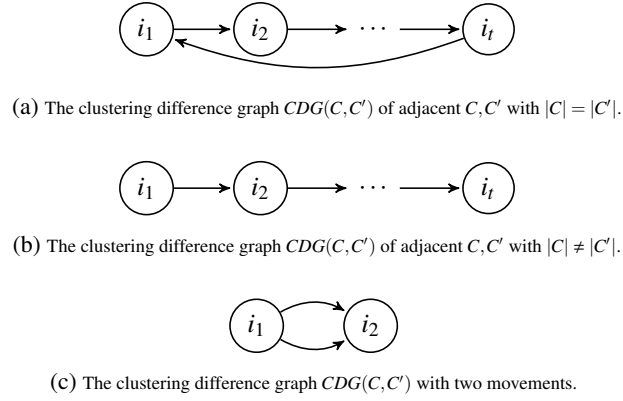

We dedicate Section \ref{sec-proofofedges} to the proof of this theorem and corollary. Figure \ref{fig:edgeofgravitypolytope} shows two vertex clusterings of $\mathcal{P}^{=}$ which are connected by an edge. The site vector of the plotted power diagram lies on the boundary of the normal cones of both vertices. As one can see, the data points on the boundary of the cells, which exist due to Corollary \ref{boundedshapepartitionpowerdiagram}, move to the other cell.

\begin{figure}[h]
\begin{center}
\scalebox{0.75}{
\begin{tikzpicture}
	\coordinate (A1) at (1,0);%
    \coordinate (A2) at (1,-1);%
    \coordinate (A3) at (2,0);%
    \coordinate (A4) at (0,-1);%
    \coordinate (B1) at (-2,-0);%
    \coordinate (B2) at (-2,-2);%
    \coordinate (B3) at (-3,-1);%
    \coordinate (B4) at (-3,1);%
    \coordinate (C1) at (-1,2);%
    \coordinate (C2) at (-1,3);%
    \coordinate (C3) at (0,3);%
    \coordinate (C4) at (1,2);%
    
     \fill[black!10] (A1) -- (A3) -- (A2) -- (A4) -- cycle;
     \fill[blue!10] (B1) -- (B2) -- (B3) -- (B4) -- cycle;
     \fill[red!10] (C1) -- (C2) -- (C3) -- (C4) -- cycle;
    
    \draw (0,0) -- (1.5,3);
    \draw (0,0) -- (0,-3);
    \draw (0,0) -- (-3.5,1.1666666667);
    
     \fill (A1) circle (0.1);
     \fill (A2) circle (0.1);
     \fill (A3) circle (0.1);
     \fill (A4) circle (0.1);
     \fill[blue] (B1) circle (0.1);
     \fill[blue] (B2) circle (0.1);
     \fill[blue] (B3) circle (0.1);
     \fill[blue] (B4) circle (0.1);
     \fill[red] (C1) circle (0.1);
     \fill[red] (C2) circle (0.1);
     \fill[red] (C3) circle (0.1);
     \fill[red] (C4) circle (0.1);
\end{tikzpicture}
}
\hspace*{1cm}
\scalebox{0.75}{
\begin{tikzpicture}
	\coordinate (A1) at (1,0);%
    \coordinate (A2) at (1,-1);%
    \coordinate (A3) at (2,0);%
    \coordinate (A4) at (0,-1);%
    \coordinate (B1) at (-2,0);%
    \coordinate (B2) at (-2,-2);%
    \coordinate (B3) at (-3,-1);%
    \coordinate (B4) at (-3,1);%
    \coordinate (C1) at (-1,2);%
    \coordinate (C2) at (-1,3);%
    \coordinate (C3) at (0,3);%
    \coordinate (C4) at (1,2);%
    
    \fill[black!10] (A1) -- (A2) -- (A3) -- (C4) -- cycle;
	\fill[blue!10] (B1) -- (A4) -- (B2) -- (B3) -- cycle;
	\fill[red!10] (C1) -- (C3) -- (C2) -- (B4) -- cycle;
	
    \draw (0,0) -- (1.5,3);
    \draw (0,0) -- (0,-3);
    \draw (0,0) -- (-3.5,1.1666666667);
    
     \fill (A1) circle (0.1);
     \fill (A2) circle (0.1);
     \fill (A3) circle (0.1);
     \fill[blue] (A4) circle (0.1);
     \fill[blue] (B1) circle (0.1);
     \fill[blue] (B2) circle (0.1);
     \fill[blue] (B3) circle (0.1);
     \fill[red] (B4) circle (0.1);
     \fill[red] (C1) circle (0.1);
     \fill[red] (C2) circle (0.1);
     \fill[red] (C3) circle (0.1);
     \fill (C4) circle (0.1);
\end{tikzpicture}
}
\caption{Two clusterings whose clustering vectors are adjacent vertices.}
\label{fig:edgeofgravitypolytope}
\end{center}
\end{figure}
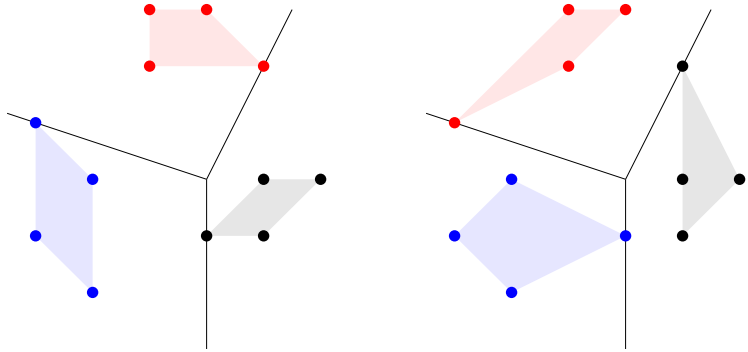

\subsection*{Proof-of-Concept Computations}

Figure \ref{fig:twopartitions} illustrates two clusterings of 27 data points in $\R^2$. These clusterings were computed by running the $k$-means algorithm 50 times with three random sites in the beginning. In every iteration the $k$-means algorithm computes a LSA to the current sites and updates each site as the arithmetic mean of the points in the cluster. This is repeated until the clustering does not change anymore. Note that the $k$-means algorithm is deterministic, but its result depends on the choice of the initial sites. Whereas the clustering on the left-hand side (except for permutation of the colors) was the output in 25 of the 50 runs, the one on the right-hand side was returned only once. Intuitively, one sees that the left-hand clustering captures the structure of the data better than the one on the right-hand side. 

\begin{figure}
\centering
\includegraphics[scale=0.38]{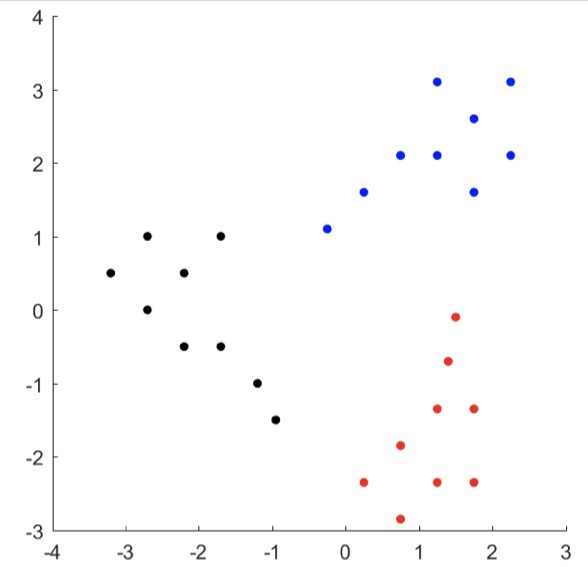}
\hspace*{0.3cm}
\includegraphics[scale=0.38]{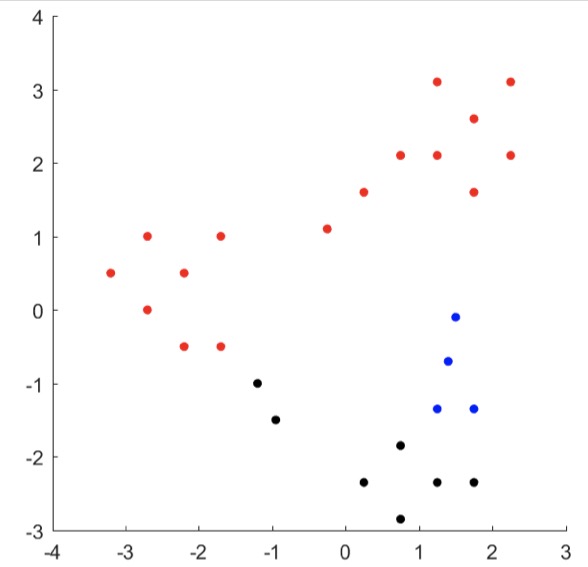}
\caption{Two vertex clusterings of $\mathcal{P}$ of a data set in $\R^2$.}
\label{fig:twopartitions}
\end{figure}

This fits with our quantitative measure: Both volumes of the clustering on the left-hand side are higher ($\mu_{\pm} \approx 0.0076$, $\mu_{=} \approx 0.041$) than the ones of the clustering on the right-hand side ($\mu_{\pm} < 0.0001$, $\mu_{=} < 0.01$).
The volumes of the respective normal cones were computed with MATLAB using the function \emph{Volume Computation of Convex Bodies} of \cite{matlab-volume} with an error tolerance of 0.001. The different volumes can also be verified by computing the edges of a respective normal cone and projecting the edges (which are (normalized) site vectors themselves) to their $d$-dimensional components for each of the $k$ sites. Each site of a site vector in the normal cone is located in the convex hull of the corresponding sites of the edges of the normal cone.


Figures \ref{fig:sitesareas} and \ref{fig:sitesareassingle} illustrate the three areas of the $2$-dimensional sites for the respective clusterings w.r.t. $\mathcal{P}$ and $\mathcal{P}^{=}$. Because of the invariance of power diagrams under scaling of sites (Proposition \ref{scalepowerdiagram}), we have to choose a fixed scaling of the $6$-dimensional site vectors to obtain a meaningful visualization. Here the site vectors are scaled to Euclidean norm 4.

\begin{figure}
\centering
\includegraphics[scale=0.38]{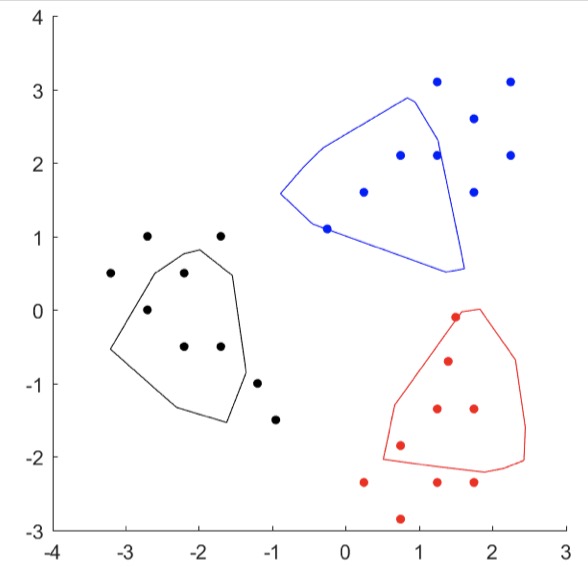}
\hspace*{0.3cm}
\includegraphics[scale=0.38]{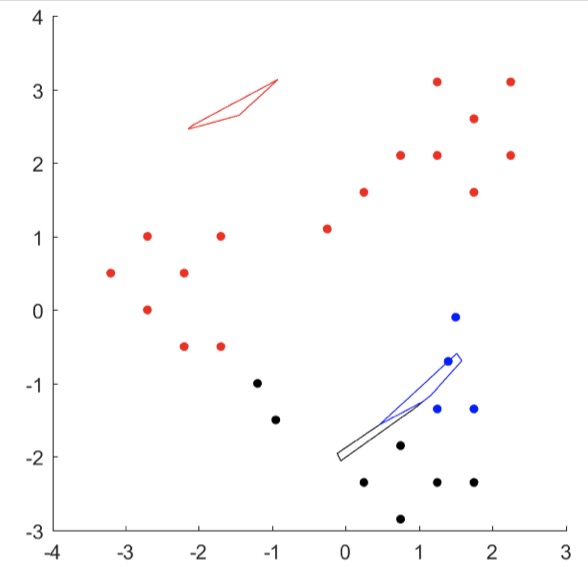}
\caption{Areas of the sites for site vectors in the vertex clustering's normal cone w.r.t. $\mathcal{P}$.}
\label{fig:sitesareas}
\end{figure}

\begin{figure}
\centering
\includegraphics[scale=0.38]{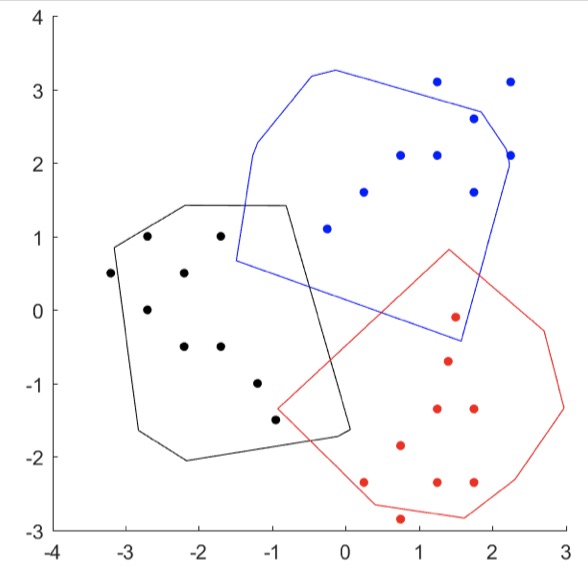}
\hspace*{0.3cm}
\includegraphics[scale=0.38]{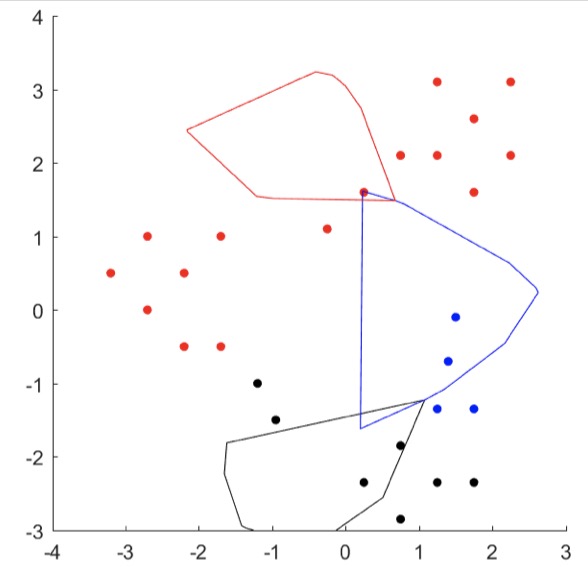}
\caption{Areas of the sites for site vectors in the vertex clustering's normal cone w.r.t. $\mathcal{P}^{=}$.}
\label{fig:sitesareassingle}
\end{figure}

One sees that the areas on the left-hand side are larger than the ones on the right-hand side, especially for the all-shape partition polytope (Figure \ref{fig:sitesareas}). Note that these areas do not mean that one can choose three sites arbitrarily within the three areas and obtain an optimal site vector for the respective clustering. Instead, for any arbitrary site within one of these areas, there are sites in the respective other areas such that the corresponding site vector is inside the normal cone of the clustering. Further, these areas contain sites of a representative for all equivalence classes of site vectors in the normal cone. For the single-shape partition polytope, these areas depict representatives of all site vectors for which the respective clusterings are optimal constrained LSAs, see Figure \ref{fig:sitesareassingle}.

\subsection{Stability of Site Vectors}\label{subsec-ballincone}

\noindent{\bf Overview.} We observed that clusterings with high volume are good for two reasons. They are likely to be computed for random sites and any site vector can be perturbed significantly without changing the clustering. Given $w(C)$ and its normal cone, we now characterize a \textbf{most stable} site vector inducing $C$. After introducing our notion of stability, which depends on the choice of a $p$-norm, we provide an optimization problem whose optimal solution gives us a site vector with the highest possible stability for this clustering. Moreover, we present how the optimal solutions for different $p$-norms are connected and how one can obtain an approximate solution for any $p$-norm by using the Euclidean norm.

\vspace*{0.2cm}


\begin{definition}[Stability of a Site Vector]\label{stablesites}
Let $p \in [1;\infty]$. For $a \in B(N_{\mathcal{P}^{\pm}}(w(C))) \subseteq \Sphere^{d \cdot k}$ the \textbf{BHR stability of the site vector $a$ w.r.t. $p$} is
\begin{equation*}
\tau^p_{\pm}(a) := \max\{ \delta > 0 \ | \ \bar{a} \in N_{\mathcal{P}^{\pm}}(w(C)) \text{ for all } \norm{a - \bar{a}}_p \leq \delta\}.
\end{equation*}
For $a \in B(N_{\mathcal{P}^{=}}(w(C))) \subseteq \Sphere^{d \cdot k}$ we call
\begin{equation*}
\tau^p_{=}(a) := \max\{ \delta > 0 \ | \ \bar{a} \in N_{\mathcal{P}^{=}}(w(C)) \text{ for all } \norm{a - \bar{a}}_p \leq \delta\}
\end{equation*}
the \textbf{LSA stability of the site vector $a$ w.r.t. $p$}.
\end{definition}

The BHR and LSA stability, $\tau^p_{\pm}$ and $\tau^p_{=}$, measure how much we can perturb the site vector $a$ within the respective normal cone w.r.t. the $p$-norm without changing the induced clustering. Of special interest are $p = 2, \infty$. Whereas the Euclidean norm weighs the perturbation of all sites equally, the infinity norm considers the highest possible perturbation of one individual site. We highlight the differences after stating our main theorem.

Note that the above definition can be extended to all site vectors by inserting the corresponding representative of the equivalence class into $\tau^p_{\star}$, $\star \in \{ \pm , = \}$. Geometrically, a ``most stable'' site vector w.r.t $p$ can be described by dropping a $p$-norm unit ball into the normal cone with $0 \in \R^{d \cdot k}$ as gravity center and computing where it gets stuck due to being blocked by the facets of the cone. The center of this unit ball then is a vector which lies ``most centrally'' within the normal cone. Figure \ref{fig:facetblockinggravitypolytope} illustrates two examples of this approach.
\begin{figure}
\centering
\begin{tikzpicture}
	\coordinate (H1) at (-2/2.2360679,1/2.2360679);
	\coordinate (H2) at (3/3.1622776,1/3.1622776);
	\draw[name path=hilfea, white] (H1) -- ($ (H1) + (2,4) $);
	\draw[name path=hilfeb, white] (H2) -- ($ (H2) + (-1.5,4.5) $);
	\path[name intersections={of=hilfea and hilfeb}];
	\coordinate (Z) at (intersection-1);
	
	\coordinate (O) at (0,0);
	\coordinate (A) at (3,6);
	\coordinate (B) at (-2,6);
	
	\fill[blue!15] (O) -- (A) -- (B) -- cycle;
	\draw[blue] (O) -- (A);
	\draw[blue] (O) -- (B);
	\fill[blue] (O) circle (0.05) node[left, blue]{$0$};
	
	\fill[red!15] (Z) circle [radius=1cm];
	\fill[red] (Z) circle (0.05) node[right] {$z^{(2)}$};
	\draw[red] (Z) circle [radius=1cm];	
\end{tikzpicture}
\hspace{1.3cm}
\begin{tikzpicture}
	\coordinate (O) at (0,0);
	\coordinate (A) at (3,6);
	\coordinate (B) at (-2,6);
	
	\fill[blue!15] (O) -- (A) -- (B) -- cycle;
	\draw[blue] (O) -- (A);
	\draw[blue] (O) -- (B);
	\fill[blue] (O) circle (0.05) node[left, blue]{$0$};
	
	\coordinate (Z) at (1/5,17/5); 
	\coordinate (1) at (6/5,12/5);
	\coordinate (2) at (6/5,22/5);
	\coordinate (3) at (-4/5,22/5);
	\coordinate (4) at (-4/5,12/5);
	
	\fill[red!15] (1) -- (2) -- (3) -- (4) -- cycle;
	\fill[red] (Z) circle (0.05) node[right] {$z^{(\infty)}$};
	\draw[red] (1) -- (2) -- (3) -- (4) -- cycle;
	
\end{tikzpicture}
\caption{Unit balls with centers $z^{(2)}$ and $z^{(\infty)}$ blocked by facets of the normal cone.}
\label{fig:facetblockinggravitypolytope}
\end{figure}
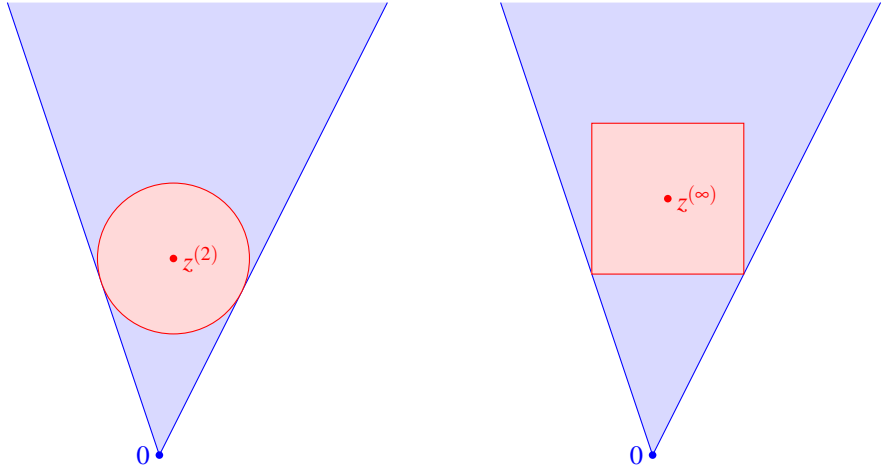


\begin{theorem}\label{solveunitballincone}
Let $p \in[1;\infty]$, $\star \in \{\pm,=\}$ and $w(C) \in \mathcal{P}^{\star}$ be a vertex with incident edges $v_1,\dots,v_t$. Then the optimal solution of the following optimization problem yields a site vector inducing the clustering $C$ with highest possible stability w.r.t. $p$.
\begin{equation} \label{optproblemgravitypolyhedron}
\begin{array}{rll}
\min & \norm{z}_2^2  & \\
\text{s.t.} & v_j^T z \leq \gamma^p_j &   \, \forall \, j \in [t], \\
 & z \in \R^{d \cdot k} & 
\end{array}
\end{equation}

\noindent with $\gamma^p_1,\dots,\gamma^p_t$ being the optimal objective values of the problems
\begin{equation} \label{optproblemgravitydistance}
\begin{array}{rl}
\min & v_j^T z \\
\text{s.t.} & z \in \B^p_1(0), \\
 & z \in \R^{d \cdot k}
\end{array}
\end{equation}
for each $j \in [t]$. Here $\B^p_{\lambda}(b) := \{ y \in \R^{d \cdot k} \ | \ \norm{y - b}_p \leq \lambda \}$ denotes the closed $p$-norm ball with center $b \in \R^{d \cdot k}$ and radius $\lambda > 0$.
\end{theorem}

Minimizing the squared Euclidean norm corresponds to dropping the $p$-norm unit ball into the cone with the origin as gravity center that attracts the ball. The constraints ensure that the ball remains inside the cone. We postpone a proof of the theorem to Section \ref{sec-proofoptproblem}.

The stability of an optimal solution $z^{(p)}$ of (\ref{optproblemgravitypolyhedron}) is $\tau^p_{\star}(z^{(p)}) = \frac{1}{\norm{z^{(p)}}_2}$. Problem (\ref{optproblemgravitypolyhedron}) is a quadratic optimization problem with linear constraints and the auxiliary problems (\ref{optproblemgravitydistance}) are linear optimization problems over a convex set. 
The edges $v_j$ encode single (cyclical) movements for all $j \in [t]$ (Theorem \ref{boundedshapepartitionedges} and Corollary \ref{singleshapepartitionedges}). Note that there might be exponentially many edges and, thus, we might have to solve exponentially many auxiliary problems.

Problem (\ref{optproblemgravitypolyhedron}) models the facets of the normal cone blocking the unit ball. We justify our approach that the ball is blocked by facets, rather than lower-dimensional faces of the cone, with a short example in Section \ref{sec-proofoptproblem}.

Fix $p \in [1;\infty)$ and let $z^{(p)} :=(z_1^T,\dots,z_k^T)^T \in \R^{d \cdot k}$ be an optimal solution of (\ref{optproblemgravitypolyhedron}). Then we can perturb one site, say $z_1 \in \R^d$, within a $p$-norm ball with radius 1 without changing the clustering. For $0 < \delta < k^{- \frac{1}{p}}$ and $\tilde{z} :=(\tilde{z}_1^T,\dots,\tilde{z}_k^T)^T \in \R^{d \cdot k}$ with $\tilde{z}_i \in \B^p_{\delta}(z_i) \subseteq \R^d$ for all $i \in [k]$, we obtain
\begin{equation*}
\norm{\tilde{z} - z^{(p)}}_p =  (\sum_{i = 1}^{k} \norm{\tilde{z}_i - z_i}_p^p )^{\frac{1}{p}} \leq (\sum_{i = 1}^{k} \delta^p )^{\frac{1}{p}} = \delta \cdot k^{\frac{1}{p}} < (\frac{1}{k})^{\frac{1}{p}} \cdot k^{\frac{1}{p}} = 1.
\end{equation*}

Thus, $\tilde{z} \in \interior(\B^p_1(z^{(p)})) \subseteq \interior(N_{\mathcal{P}^{\star}}(w(C)))$, i.e. we can perturb each site within a $p$-norm ball with e.g. radius $\delta := (k+1)^{- \frac{1}{p}} < k^{- \frac{1}{p}}$ without changing the clustering.

If $p = \infty$, we can even choose $0 < \delta < 1$, because then
\begin{equation*}
\norm{\tilde{z} - z^{(p)}}_{\infty} = \max\{ \norm{\tilde{z}_i - z_i}_{\infty} \ | \ i \in [k]\} \leq \delta < 1.
\end{equation*}

Next, we show how to obtain a feasible approximate solution of (\ref{optproblemgravitypolyhedron}). It is well-known that for all $p,q \in [1;\infty]$ there is a positive constant $c_{p,q} > 0$ only depending on $p,q$ and the dimension of the space such that $\norm{x}_p \leq c_{p,q} \norm{x}_q$ for all $x \in \R^{d \cdot k}$.

\begin{theorem} \label{pqnormegal}
Let $p \in [1;\infty]$, $\star \in \{\pm , = \}$ and $z^{(p)} \in \R^{d \cdot k}$ be an optimal solution of (\ref{optproblemgravitypolyhedron}). For all $q \in [1;\infty]$ the vector $z' = c_{p,q} z^{(p)}$ satisfies $\B^q_1(z') \subseteq N_{\mathcal{P}^{\star}}(w(C))$.
Moreover, its objective value satisfies $\norm{z'}^2_2 \leq \max\{c_{p,q},c_{q,p}\}^2 \norm{z^{(q)}}^2_2$ where $z^{(q)} \in \R^{d \cdot k}$ is the optimal solution of (\ref{optproblemgravitypolyhedron}) w.r.t. the $q$-norm.
\end{theorem}

Note that an upper bound on the objective value of the approximate solution $z'$ implies a lower bound on its stability $\tau^p_{\star}(z')$ and that $d(z',z^{(p)}) = 0$, i.e. $z'$ and $z^{(p)}$ define the same separating power diagram.
If we choose $p = 2$, then the auxiliary problems (\ref{optproblemgravitydistance}) of Theorem \ref{solveunitballincone} have optimal objective values $\gamma_j = - \norm{v_j}_2$ for all $j \in [t]$. Problem (\ref{optproblemgravitypolyhedron}) then reduces to
\begin{equation} \label{optproblemeuclidean}
\begin{array}{rll}
\min & \norm{z}_2^2  & \\
\text{s.t.} & v_j^T z \leq - \norm{v_j}_2 &   \, \forall \, j \in [t], \\
 & z \in \R^{d \cdot k}. & 
\end{array}
\end{equation}

After computing an optimal solution $z^{(2)}$ of (\ref{optproblemeuclidean}), we get a site vector $z' = c_{2,p} z^{(2)}$ for any arbitrary $p \in [1;\infty]$ whose norm can be bounded from above by the norm of the most stable (w.r.t. $p$) site vector $z^{(p)}$ and a constant factor only depending on $p$ (Theorem \ref{pqnormegal}). Hence, we obtain a provable approximation.

\subsection*{Proof-of-Concept Computations}

We use Figures \ref{fig:sitesone2norm} -- \ref{fig:sitesinftynormsingle} to illustrate computational results for the the two clusterings of Figure \ref{fig:twopartitions}. We compute the optimal sites (crosses in respective colors), following the programs in Theorem \ref{solveunitballincone}, w.r.t. $p = 2, \infty$ for the all-shape and single-shape partition polytopes. The $6$-dimensional optimal solutions of (\ref{optproblemgravitypolyhedron}) were again scaled to Euclidean norm 4.
Figures \ref{fig:sitesone2norm} and \ref{fig:sitesone2normsingle} depict the area of possible perturbation without changing the clustering when only perturbing the first (black) site and keeping the other two sites fixed.

\begin{figure}
\centering
\includegraphics[scale=0.38]{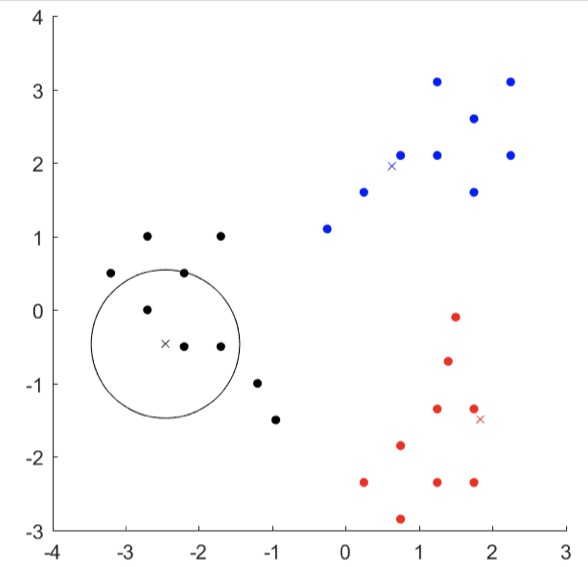}
\hspace*{0.3cm}
\includegraphics[scale=0.38]{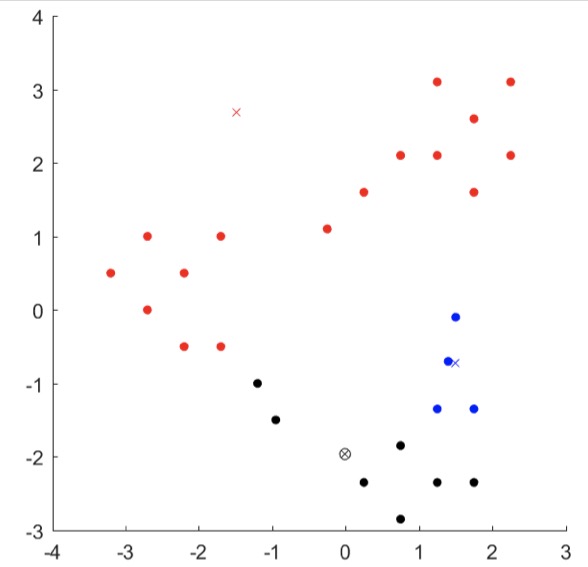}
\caption{The optimal sites w.r.t. $\mathcal{P}$ for $p=2$ with perturbation of the first site.}
\label{fig:sitesone2norm}
\end{figure}

\begin{figure}
\centering
\includegraphics[scale=0.38]{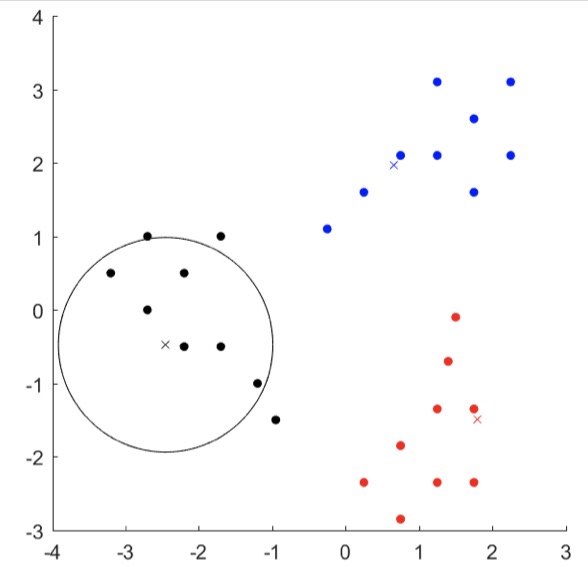}
\hspace*{0.3cm}
\includegraphics[scale=0.38]{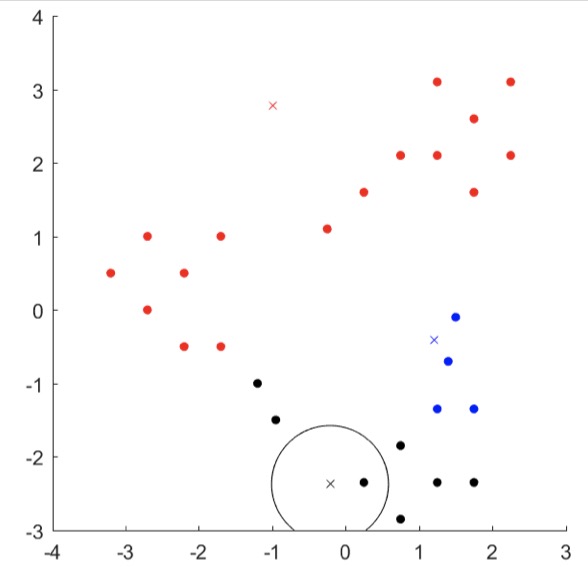}
\caption{The optimal sites w.r.t. $\mathcal{P}^{=}$ for $p=2$ with perturbation of the first site.}
\label{fig:sitesone2normsingle}
\end{figure}

In Figures \ref{fig:sites2norm} -- \ref{fig:sitesinftynormsingle}, all sites can be perturbed simultaneously within the drawn $p$-norm balls without changing the clustering.
The different sizes of the $p$-norm balls are due to scaling of the optimal solutions of (\ref{optproblemgravitypolyhedron}). Note that scaling the sites does not change the clustering, but of course, the scaling does affect the radius of the $p$-norm ball. In fact, the radius directly corresponds to the stability of the depicted optimal sites. The BHR stability of the left-hand clustering is $\tau^2_{\pm}(z^{(2)}) \approx 0.253$ and $\tau^{\infty}_{\pm}(z^{(\infty)}) \approx 0.139$, whereas for the right-hand clustering we get $\tau^{2}_{\pm}(z^{(2)}) \approx 0.018$ and $\tau^{\infty}_{\pm}(z^{(\infty)}) \approx 0.009$. Note that on the right-hand side of Figures \ref{fig:sites2norm} and \ref{fig:sitesinftynorm}, the balls around each cross that indicate the area of possible perturbation are so small that they become hard to see.

\begin{figure}
\centering
\includegraphics[scale=0.38]{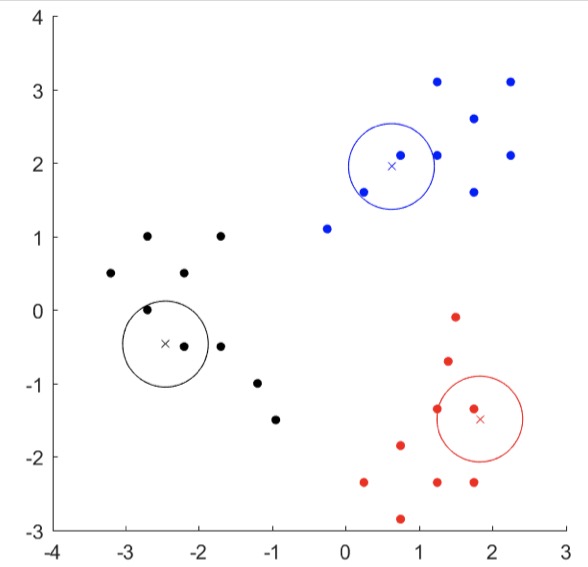}
\hspace*{0.3cm}
\includegraphics[scale=0.38]{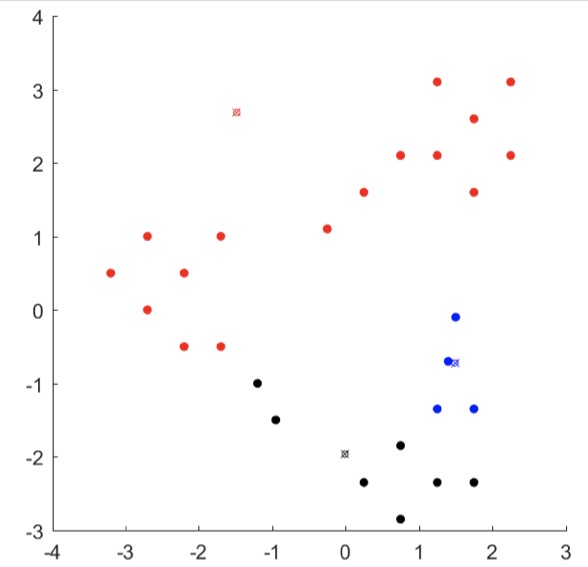}
\caption{The optimal sites w.r.t. $\mathcal{P}$ for $p = 2$ and the areas of possible perturbations.}
\label{fig:sites2norm}
\end{figure}

\begin{figure}
\centering
\includegraphics[scale=0.38]{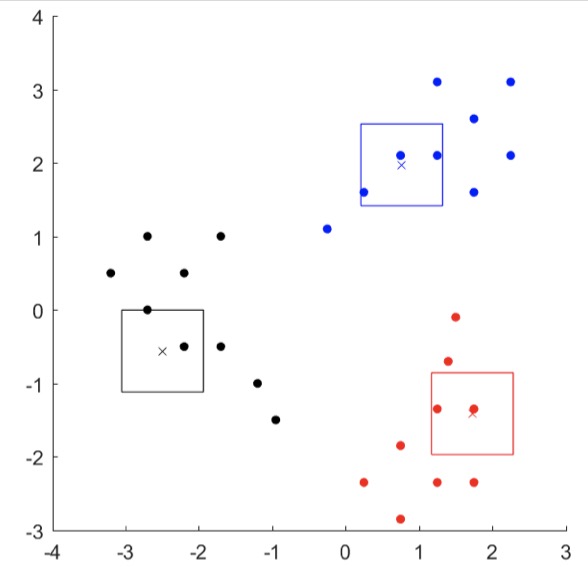}
\hspace*{0.3cm}
\includegraphics[scale=0.38]{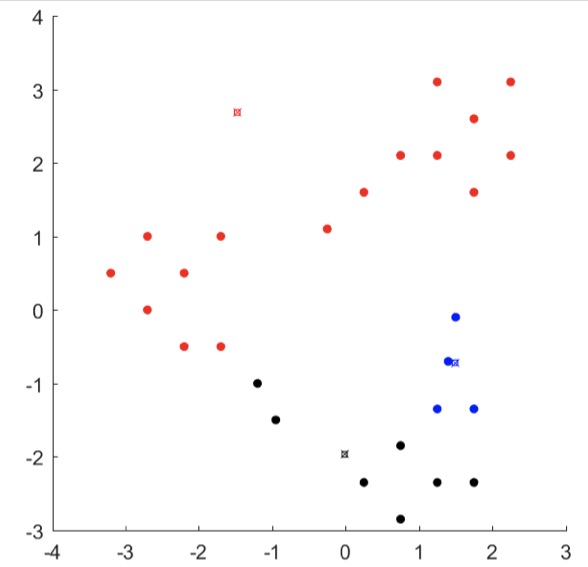}
\caption{The optimal sites w.r.t. $\mathcal{P}$ for $p = \infty$ and the areas of possible perturbations.} 
\label{fig:sitesinftynorm}
\end{figure}

Figures \ref{fig:sites2normsingle} and \ref{fig:sitesinftynormsingle} depict the areas of perturbation for the single-shape partition polytope. The difference between left-hand and right-hand clustering of the LSA stability is not as large as in the previous examples, since the cluster sizes are fixed. We obtain $\tau^2_{=}(z^{(2)}) \approx 0.365 $, $\tau^{\infty}_{=}(z^{(\infty)}) \approx 0.163$ (left) and $\tau^2_{=}(z^{(2)}) \approx 0.199$, $\tau^{\infty}_{=}(z^{(\infty)}) \approx 0.100$ (right). Nevertheless, one sees that the sites on the left-hand side can be perturbed more than the ones on the right-hand side.

\begin{figure}
\centering
\includegraphics[scale=0.38]{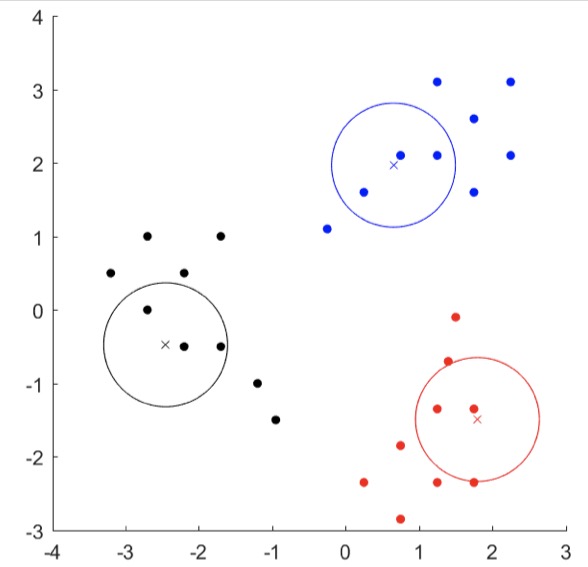}
\hspace*{0.3cm}
\includegraphics[scale=0.38]{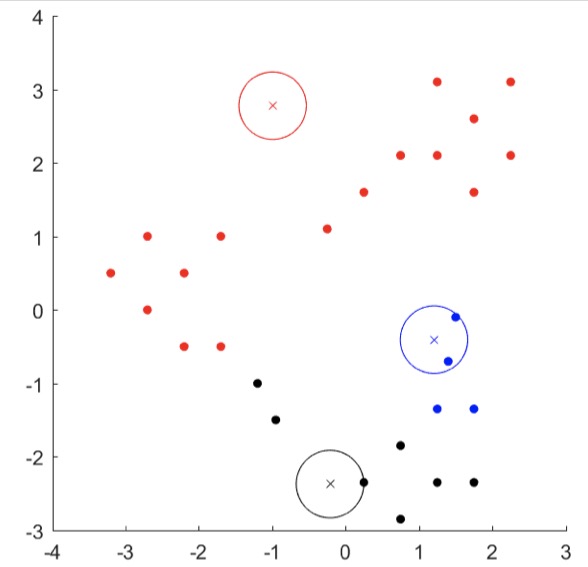}
\caption{The optimal sites w.r.t. $\mathcal{P}^{=}$ for $p = 2$ and the areas of possible perturbations.}
\label{fig:sites2normsingle}
\end{figure}
 
\begin{figure}
\centering
\includegraphics[scale=0.38]{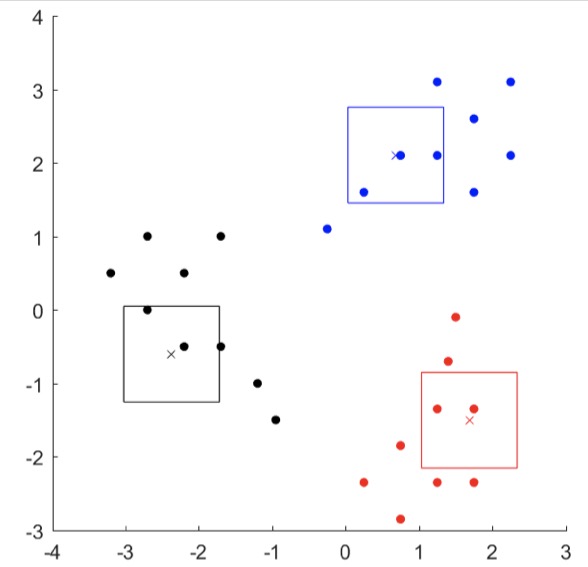}
\hspace*{0.3cm}
\includegraphics[scale=0.38]{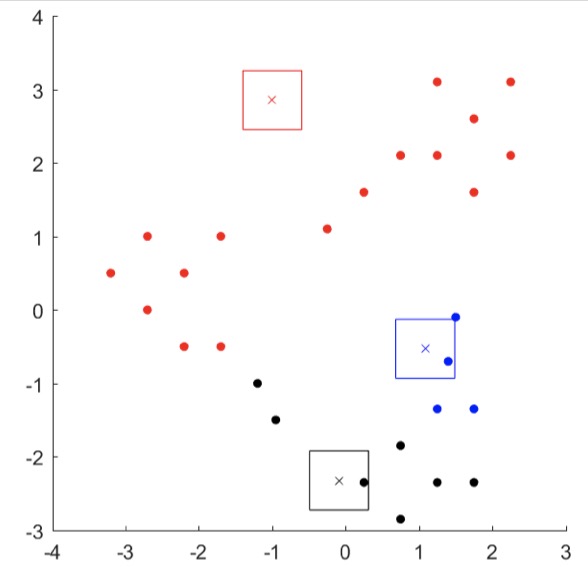}
\caption{The optimal sites w.r.t. $\mathcal{P}^{=}$ for $p = \infty$ and the areas of possible perturbations.}
\label{fig:sitesinftynormsingle}
\end{figure}

\section{Proofs}

In this section, we present proofs of theorems stated in Sections \ref{subsec-largecones} -- \ref{subsec-ballincone}.
We will prove equivalent statements for $\mathcal{P}^{\pm}$ and $\mathcal{P}^{=}$ just for the bounded-shape case, since they work analogously. It would suffice to replace ``$\pm$'' by ``$=$''.

\subsection{Proofs for Section \ref{subsec-largecones}}\label{sec-proofofmeasure}

\proof{Proof of Theorem \ref{measure}.}
A clustering $C$ is not a vertex clustering if and only if $\interior(N_{\mathcal{P}^{\pm}}(w(C)))=\emptyset$, yielding vol$(N_{\mathcal{P}^{\pm}}(w(C))) = 0$ and $\mu_{\pm}(C) = 0$. So let $C$ be a vertex clustering.

If $a \in N_{\mathcal{P}^{\pm}}(w(C))$, then $\frac{1}{\norm{a}_2} a \in N_{\mathcal{P}^{\pm}}(w(C)) \cap \Sphere^{d \cdot k}$. Suppose we are given a site vector $z \in \interior(N_{\mathcal{P}^{\pm}}(w(C))) \cap \Sphere^{d \cdot k}$. Then $z$ is contained in the relative interior of $B(N_{\mathcal{P}^{\pm}}(w(C)))$.
The base of a cone is spherically convex, so for all $a \in B(N_{\mathcal{P}^{\pm}}(w(C)))$ the geodesic ``connecting'' $z$ and $\frac{1}{\norm{a}_2} a$, i.e. the $\gamma$ satisfying $L(\gamma) = d(z,a)$, is contained in $B(N_{\mathcal{P}^{\pm}}(w(C)))$.
The area of $B(N_{\mathcal{P}^{\pm}}(w(C)))$, which is equal to vol$(N_{\mathcal{P}^{\pm}}(w(C)))$, can be computed by an integral over all directions (geodesics through $z$) of $d(z,a)$ with $a \in B(N_{\mathcal{P}^{\pm}}(w(C)))$ maximizing $d(z,\cdot)$ along the direction of $\gamma$. In fact, starting at $z$, we can walk into any direction along the unit sphere for a small distance and still stay in the interior of the normal cone. 
In particular, for every $a \in \bd(B(N_{\mathcal{P}^{\pm}}(w(C)))) \subseteq \bd(N_{\mathcal{P}^{\pm}}(w(C)))$, there is a geodesic $\gamma$ satisfying $L(\gamma) = d(z,a) > 0$, and for every geodesic $\gamma$ starting at $z$ we have
\begin{equation*}
\max\{ d(z,a) \ | \ a \in \gamma, \, a \in B(N_{\mathcal{P}^{\pm}}(w(C))) \} = d(z,\bar{a}) > 0
\end{equation*}
with $\{\bar{a}\} = \bd(B(N_{\mathcal{P}^{\pm}}(w(C)))) \cap \gamma$. We obtain
\begin{equation*}
\text{vol}(N_{\mathcal{P}^{\pm}}(w(C))) = \int\limits_{\bd(B(N_{\mathcal{P}^{\pm}}w(C))))} d(z,a) \ da ,
\end{equation*}
which, together with the fact that $\vol(\R^{d \cdot k})$ is equal to the area of the surface of $\Sphere^{d \cdot k}$, yields representation (\ref{integralmeasure}). 

Let $K \subseteq \R^{d \cdot k}$ be a cone with $\vol(K) = \vol(N_{\mathcal{P}^{\pm}}(w(C)))$ and consider a diffeomorphism $f: B(N_{\mathcal{P}^{\pm}}(w(C))) \mapsto B(K)$. Then $f(z)$ satisfies $d(f(z),a) > 0$ for all $a \in \bd(B(K))$, as open sets (in $B(N_{\mathcal{P}^{\pm}}(w(C)))$) are mapped to open sets (in $B(K)$) and vice versa. Furthermore, as $\vol(K) = \vol(N_{\mathcal{P}^{\pm}}(w(C)))$ and the bases $B(K)$ and $B(N_{\mathcal{P}^{\pm}}(w(C)))$ are spherically convex, we can perturb $f(z)$ and $z$ in equal measure without leaving the bases of the respective cones. This shows that $\mu_{\pm}$ as a quality measure is well-defined. \hfill \qed
\endproof

\subsection{Proofs for Section \ref{subsec-edgesofpolytope}}\label{sec-proofofedges}

In this section, we work towards a proof of Theorem \ref{boundedshapepartitionedges}. First, we need a few auxiliary lemmas.

Fix $C,C'$ to be the clusterings of Theorem \ref{boundedshapepartitionedges}. Let $CDG(C,C')$ be their clustering difference graph and let $\mathcal{M}$ be the set of movements corresponding to a decomposition of $CDG(C,C')$ into cycles and paths. Denote by $\mathcal{M}_{\geq 2} \subseteq \mathcal{M}$ the set of movements corresponding to cycles and non-trivial paths, i.e. paths (and cycles) containing at least two edges.

\begin{lemma}\label{vectorscollinear}
Let $a \in \R^{d \cdot k}$ such that $a^T w(C) = a^T w(C') > a^T w$ for all other vertices $w \in \mathcal{P}^{\pm} \setminus \{w(C) , w(C')\}$. Then $a^T w(M) = 0$ for all $M \in \mathcal{M}$.
\end{lemma}

\proof{Proof.}
Let $M \in \mathcal{M}$ and suppose $a^T w(M) > 0$. Applying $M$ to $C$ yields a (feasible) clustering $\bar{C}$ with $a^T w(\bar{C}) = a^T (w(C) + w(M)) > a^T w(C)$ which is a contradiction. For $a^T w(M) < 0$ we obtain the same contradiction by applying $M^{-1}$ to $C'$, as this would yield a (feasible) clustering $\bar{C}$ with
\begin{equation*}
a^T w(\bar{C}) = a^T (w(C') + w(M^{-1})) = a^T w(C') - a^T w(M) > a^T w(C').
\end{equation*}
So $a^T w(M) = 0$. \hfill \qed
\endproof

This implies that all movements in $\mathcal{M}$ operate on the same clusters.

\begin{lemma}\label{operateonsameclusters}
There exists an index subset $I:= \{i_1,\dots,i_t\} \subseteq [k]$ such that all paths and cycles in $CDG(C,C')$ corresponding to (cyclical) movements in $\mathcal{M}$ contain all nodes in $I$ and no other.
\end{lemma}

\proof{Proof.}
By Lemma \ref{vectorscollinear}, the vectors of all movements in $\mathcal{M}$ are collinear. Recall that $X$ consists of distinct, non-zero data points. So, if movements would move items between different subsets of clusters, their corresponding vectors would leave different components of the clustering vectors unchanged. But then the vectors of these movements would not be collinear.

Let $M_1: C_{i_1} \rightarrow C_{i_2} \rightarrow \cdots \rightarrow C_{i_t}$ be a movement in $\mathcal{M}$ for some $2 \leq t \leq k$. Note that if $t \geq 3$ and $i_1 = i_t$, then $M_1$ is a cyclical movement. Suppose there exists a movement $M_2$ whose corresponding path in the CDG contains an edge $(i_l,i_j)$ with $j \not= l+1$ (cyclical indexing). Then the paths in $CDG(C,C')$ take one of the forms illustrated in Figure \ref{fig:cyclicalexchangeorder}. Figure \ref{fig:cyclicalexchangeorder1} shows the path corresponding to $M_1$.

\begin{figure}[h!]
\centering
\begin{subfigure}{14cm}
\centering
\begin{tikzpicture}[->,>=stealth',shorten >=1pt,auto,node distance=1.5cm,semithick]
	\node[circle, draw] (1) {$i_1$};
	\node[circle, draw] (2) [right of=1] {$i_2$};
	\node (3) [right of=2] {$\cdots$};
	\node[circle, draw] (i) [right of=3] {$i_l$};
	\node[ellipse, draw] (l) [right of=i] {$i_{l+1}$};
	\node (4) [right of=l] {$\cdots$};
	\node[circle, draw] (j) [right of=4] {$i_j$};
	\node (5) [right of=j] {$\cdots$};
	\node[circle, draw] (s) [right of=5] {$i_t$};
	
	\path 	(1) edge (2)
			(2) edge (3)
			(3) edge (i)
			(i) edge (l)
			(l) edge (4)
			(4) edge (j)
			(j) edge (5)
			(5) edge (s);
\end{tikzpicture}
\caption{\scriptsize Path corresponding to $M_1: C_{i_1} \rightarrow C_{i_2} \rightarrow \cdots \rightarrow C_{i_t}$.}
\label{fig:cyclicalexchangeorder1}
\end{subfigure}

\vspace*{0.5cm}

\begin{subfigure}{14cm}
\centering
\begin{tikzpicture}[->,>=stealth',shorten >=1pt,auto,node distance=1.5cm,semithick]
	\node[circle, draw] (1) {$i_1$};
	\node[circle, draw] (2) [right of=1] {$i_2$};
	\node (3) [right of=2] {$\cdots$};
	\node[circle, draw] (i) [right of=3] {$i_l$};
	\node[ellipse, draw] (l) [right of=i] {$i_{l+1}$};
	\node (4) [right of=l] {$\cdots$};
	\node[circle, draw] (j) [right of=4] {$i_j$};
	\node (5) [right of=j] {$\cdots$};
	\node[circle, draw] (s) [right of=5] {$i_t$};
	
	\path 	(1) edge (2)
			(2) edge (3)
			(3) edge (i)
			(i) edge [bend left] (j)
			(j) edge (5)
			(5) edge (s);
\end{tikzpicture}
\caption{\scriptsize Path corresponding to $M_2$ for $j > l+1$.}
\label{fig:cyclicalexchangeorder2}
\end{subfigure}

\vspace*{0.5cm}

\begin{subfigure}{14cm}
\centering
\begin{tikzpicture}[->,>=stealth',shorten >=1pt,auto,node distance=1.5cm,semithick]
	\node[circle, draw] (1) {$i_1$};
	\node[circle, draw] (2) [right of=1] {$i_2$};
	\node (3) [right of=2] {$\cdots$};
	\node[circle, draw] (j) [right of=3] {$i_j$};
	\node (4) [right of=j] {$\cdots$};
	\node[circle, draw] (i) [right of=4] {$i_l$};
	\node (5) [right of=i] {$\cdots$};
	\node[circle, draw] (s) [right of=5] {$i_t$};
	
	\path 	(1) edge (2)
			(2) edge (3)
			(3) edge (j)
			(j) edge [red] (4)
			(4) edge [red] (i)
			(i) edge [bend left, red](j);

\end{tikzpicture}
\caption{\scriptsize Cycle corresponding to $M_2$ for $j < l $ in red.}
\label{fig:cyclicalexchangeorder3}
\end{subfigure}
\vspace*{0.5cm}
\caption{Possible paths in the clustering difference graph.}
\label{fig:cyclicalexchangeorder}
\end{figure}
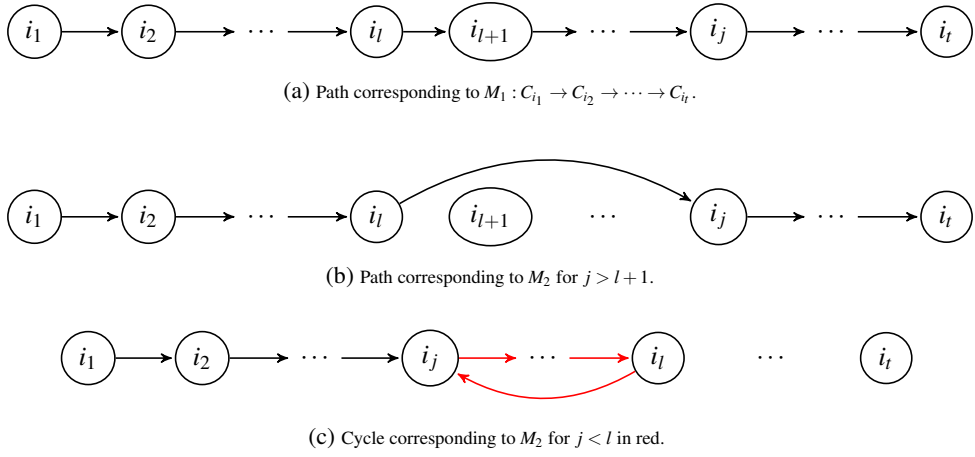

For $l+1 < j$, $M_2$ skips all clusters between $i_l$ and $i_j$, see Figure \ref{fig:cyclicalexchangeorder2}. For $l > j$ we obtain a cyclical movement skipping all clusters $C_{i_r}$ with $r < j$ or $r > l$, c.f. Figure \ref{fig:cyclicalexchangeorder3}. In both cases, there exists a (cyclical) movement that leaves different components of the clustering vectors unchanged. In particular, $M_2$ does not change cluster $i_{l+1}$, but $M_1$ does. Therefore, the vectors $w(M_1)$ and $w(M_2)$ cannot be collinear, contradicting Lemma \ref{vectorscollinear}. \hfill \qed
\endproof

\begin{lemma}\label{atmostonecyclicalexchange}
If no three points in $X$ lie on a single line, then $\abs{\mathcal{M}_{\geq 2}} \leq 1$.
\end{lemma}


\proof{Proof.}
Assume by contradiction that $\abs{\mathcal{M}_{\geq 2}} \geq 2$, i.e. there are at least two non-trivial paths or cycles in $CDG(C,C')$. All movements in $\mathcal{M}_{\geq 2}$ operate on the same clusters (Lemma \ref{operateonsameclusters}) and their vectors are collinear (Lemma \ref{vectorscollinear}). Suppose there are two movements $M_1, M_2 \in \mathcal{M}_{\geq 2}$ of the form
\begin{align*}
M_1: &\quad C_{i_1} \stackrel{x_{i_1}}\longrightarrow \cdots \stackrel{x_{i_{l-2}}}\longrightarrow C_{i_{l-1}} \stackrel{x_{i_{l-1}}}\longrightarrow C_{i_l} \stackrel{x_{i_l}}\longrightarrow C_{i_{l+1}} \stackrel{x_{i_{l+1}}}\longrightarrow \cdots \stackrel{x_{i_{t-1}}}\longrightarrow C_{i_t} , \\
M_2: &\quad C_{i_l} \stackrel{x'_{i_l}}\longrightarrow C_{i_{l+1}} \stackrel{x'_{i_{l+1}}}\longrightarrow \cdots \stackrel{x'_{i_{t-1}}}\longrightarrow C_{i_t} \stackrel{x'_{i_t}}\longrightarrow C_{i_1} \stackrel{x'_{i_1}}\longrightarrow \cdots \stackrel{x'_{i_{l-2}}}\longrightarrow C_{i_{l-1}}
\end{align*}
for some $3 \leq t \leq k$ and $l \in [t]$ (cyclical indexing). Note that, if $i_1 = i_t$ or $i_{l-1} = i_l$, then $M_1$ or $M_2$ corresponds to a cycle, respectively. In this case, the element $x'_t$, respectively $x_{i_{l-1}}$, does not exist. Both paths corresponding to $M_1$ and $M_2$ have to contain the same nodes in the same order, otherwise we would obtain a contradiction as in the proof of Lemma \ref{operateonsameclusters}.

Since every $x \in X$ is moved at most once, we know $x_{i_r} \not= x'_{i_j}$ for all $r \in [t-1]$ and $j \in [t] \setminus\{l-1\}$. Consider the part of $CDG(C,C')$ that contains both paths, c.f. Figure \ref{fig:nocyclicalexchange1}. We can construct two movements $M'_1$ and $M'_2$, as explained in Figure \ref{fig:nocyclicalexchange}.

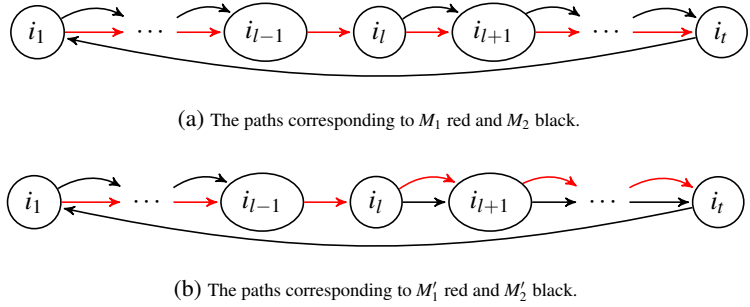
\begin{figure}[h]
\centering
\begin{subfigure}{14cm}
\centering
\begin{tikzpicture}[->,>=stealth',shorten >=1pt,auto,node distance=1.5cm,semithick]
	\node[circle, draw] (1) {$i_1$};
	\node (2) [right of=1] {$\cdots$};
	\node[ellipse, draw] (i) [right of=2] {$i_{l-1}$};
	\node[circle, draw] (i1) [right of=i] {$i_l$};
	\node[ellipse, draw] (i2) [right of=i1] {$i_{l+1}$};
	\node (3) [right of=i2] {$\cdots$};
	\node[circle, draw] (s) [right of=3] {$i_t$};
	
	\path 	(1) edge [red](2)
			(2) edge [red] (i)
			(i) edge [red] (i1)
			(i1) edge [red] (i2)
			(i2) edge [red](3)
			(3) edge [red] (s);
			
	\path	(i1) edge [bend left] (i2)
			(i2) edge [bend left] (3)
			(3) edge [bend left] (s)
			(s) edge [bend left=12] (1)
			(1) edge [bend left] (2)
			(2) edge [bend left] (i);		
\end{tikzpicture}
\caption{\scriptsize The paths corresponding to $M_1$ red and $M_2$ black.}
\label{fig:nocyclicalexchange1}
\end{subfigure}

\vspace*{0.5cm}

\begin{subfigure}{14cm}
\centering
\begin{tikzpicture}[->,>=stealth',shorten >=1pt,auto,node distance=1.5cm,semithick]
	\node[circle, draw] (1) {$i_1$};
	\node (2) [right of=1] {$\cdots$};
	\node[ellipse, draw] (i) [right of=2] {$i_{l-1}$};
	\node[circle, draw] (i1) [right of=i] {$i_l$};
	\node[ellipse, draw] (i2) [right of=i1] {$i_{l+1}$};
	\node (3) [right of=i2] {$\cdots$};
	\node[circle, draw] (s) [right of=3] {$i_t$};
	
	\path 	(1) edge [red](2)
			(2) edge [red] (i)
			(i) edge [red] (i1)
			(i1) edge (i2)
			(i2) edge (3)
			(3) edge (s);
			
	\path	(i1) edge [bend left, red] (i2)
			(i2) edge [bend left, red] (3)
			(3) edge [bend left, red] (s)
			(s) edge [bend left=12] (1)
			(1) edge [bend left] (2)
			(2) edge [bend left] (i);	
\end{tikzpicture}
\caption{\scriptsize The paths corresponding to $M'_1$ red and $M'_2$ black.}
\label{fig:nocyclicalexchange2}
\end{subfigure}
\vspace*{0.5cm}
\caption{Construction of $M'_1$ and $M'_2$ in $CDG(C,C')$.}
\label{fig:nocyclicalexchange}
\end{figure}

In Figure \ref{fig:nocyclicalexchange2}, we see that $M'_j$ equals $M_j$ for $j=1,2$ between the nodes $i_t$ and $i_{l}$. Between the $i_l$-th and the $i_t$-th cluster, $M'_1$ equals $M_2$ and $M'_2$ equals $M_1$.
The clustering we obtain after applying the movements $M'_1$ and $M'_2$ to $C$ equals the one after applying $M_1$ and $M_2$ to $C$. Thus we could replace the paths corresponding to $M_1$ and $M_2$ by the ones corresponding to $M'_1$ and $M'_2$ in the decomposition of $CDG(C,C')$. By Lemma \ref{vectorscollinear}, the vectors $w(M_1)$, $w(M_2)$, $w(M'_1)$ and $w(M'_2)$ are collinear. Let us now consider the components of $w(M_1)$ and $w(M'_1)$ corresponding to the changes of $\sigma_{i_l}$. We obtain for $w(M_1)$ and $w(M'_1)$ (in this order)
\begin{equation*}
 x_{i_{l-1}} - x_{i_l}, \quad x_{i_{l-1}} - x'_{i_l},
\end{equation*}
which are collinear. But this implies that $x_{i_{l-1}},x_{i_l}$ and $x'_{i_l}$ lie on a single line, a contradiction. Hence, $\abs{\mathcal{M}_{\geq 2}} < 2$. \hfill \qed
\endproof

In particular, $CDG(C,C')$ contains at most one cycle or path of edge-length greater than 1. 

\proof{Proof of Theorem \ref{boundedshapepartitionedges}.}
If $\abs{C} = \abs{C'}$, then the statement follows from Lemma \ref{atmostonecyclicalexchange} and the fact that $C \not= C'$ implies that $CDG(C,C')$ contains at least one cycle. So suppose $\abs{C} \not= \abs{C'}$.
If $CDG(C,C')$ contains a path $(i_1,i_2)-(i_2,i_3) - \cdots - (i_{t-1},i_t)$ with $t \geq 3$, then, by Lemma \ref{atmostonecyclicalexchange}, there does not exist any other paths or cycle. Now let $t =2$ and suppose there are three movements $M_1, M_2$ and $M_3$, which all are of the form $C_{i_1} \rightarrow C_{i_2}$.
By collinearity of the vectors $w(M_1)$, $w(M_2)$ and $w(M_3)$, the respective components corresponding to $\sigma_{i_2}$ are collinear. By definition, these entries equal the labels of the three edges in $CDG(C,C')$ -- which are collinear. This is a contradiction to the assumption that no three points in $X$ lie on a single line. \hfill \qed
\endproof

\proof{Proof of Corollary \ref{singleshapepartitionedges}.}
Note that $\abs{C} = \abs{C'}$ implies that $\mathcal{M} = \mathcal{M}_{\geq 2}$, since the CDG decomposes into cycles. Suppose there exist two cycles in $CDG(C,C')$ with corresponding cyclical movements $M_1$ and $M_2$ of the form
\begin{align*}
M_1: & \quad C_{i_1} \stackrel{x_{j_1}}\longrightarrow C_{i_2} \stackrel{x_{j_2}}\longrightarrow \cdots \stackrel{x_{j_{t-1}}}\longrightarrow C_{i_t} \stackrel{x_{j_t}}\longrightarrow C_{i_1}, \\
M_2: & \quad C_{i_1} \stackrel{x'_{j_1}}\longrightarrow C_{i_2} \stackrel{x'_{j_2}}\longrightarrow \cdots \stackrel{x'_{j_{t-1}}}\longrightarrow C_{i_t} \stackrel{x'_{j_t}}\longrightarrow C_{i_1},
\end{align*}
for some $t \geq 2$.
Again, we have $x_{j_l} \not= x'_{j_r}$ for all $l, r \in [t]$. Analogously to the proof of Lemma \ref{atmostonecyclicalexchange}, we can construct two different cycles $M'_1$ and $M'_2$ via $M'_j = M_j$ between $i_1$ and $i_t$ and $M'_1 = M_2$ and $M'_2 = M_1$ between clusters $C_{i_t}$ and $C_{i_1}$. The corresponding vectors $w(M_1)$, $w(M_2)$, $w(M'_1)$ and $w(M'_2)$ are collinear and, considering the $d$ entries corresponding to $\sigma_{i_1}$, we obtain (in this order)
\begin{equation*}
x_{j_t} - x_{j_1}, \quad x'_{j_t} - x'_{j_1}, \quad x'_{j_t} - x_{j_1}, \quad x_{j_t} - x'_{j_1},
\end{equation*}
which are collinear. This implies that $x_{j_1}$, $x'_{j_1}$, $x_{j_t}$ and $x'_{j_t}$ lie on a line, a contradiction. \hfill \qed
\endproof

\subsection{Proofs for Section \ref{subsec-ballincone}}\label{sec-proofoptproblem}

\proof{Proof of Theorem \ref{solveunitballincone}.}
Let $p \in [1;\infty]$, let $C$ be a vertex clustering and let $N_{\mathcal{P}^{\pm}}(w(C))$ be its normal cone with facets $F_1,\dots,F_t$. The ``gravity center'' at $0 \in \R^{d \cdot k}$ which attracts the $p$-norm ball can be modeled by minimizing the squared Euclidean norm $\norm{z}_2^2$. Using $dist(z,F_j)_p = \inf\{ \norm{ z - f^j }_p \ | \ f^j \in F_j \}$ to denote the distance of $z$ and the facet $F_j$ w.r.t. the $p$-norm, constraints $dist(z,F_j)_p \geq 1$ for all $j \in [t]$ guarantee that a  $p$-norm unit ball with center $z$ remains inside the normal cone $N_{\mathcal{P}^{\pm}}(w(C))$. This gives the
 optimization problem
\begin{equation} \label{optproblemgravity}
\begin{array}{rll}
\min & \norm{z}_2^2  & \\
\text{s.t.} & z \in N_{\mathcal{P}^{\pm}}(w(C)), &   \\
 & dist(z,F_j)_p \geq 1 &  \, \forall \, j \in [t], \\
 & z \in \R^{d \cdot k}. & 
\end{array}
\end{equation}

First we show that an optimal solution $z^{(p)}$ of (\ref{optproblemgravity}) has the highest stability w.r.t. $p$ among all vectors in $N_{\mathcal{P}^{\pm}}(w(C))$.

The stability of $z^{(p)}$ is equal to $\tau^p_{\pm}(z^{(p)}) = \frac{1}{\norm{z^{(p)}}_2} > 0$ by construction. Note that any site vector on the boundary of the normal cone has stability $0$. So let $z \in \interior(N_{\mathcal{P}^{\pm}}(w(C)))$ and let $\delta_z := \min\{dist(z,F_j)_p \ | \ j \in [t] \} > 0$ be the smallest $p$-norm distance to any of the facets. Then $\tau^p_{\pm}(z) = \frac{\delta_z}{\norm{z}_2}$. Moreover, the vector $\frac{1}{\delta_z} z$ is feasible for (\ref{optproblemgravity}) and, by optimality of $z^{(p)}$, we obtain
\begin{equation*}
\tau^p_{\pm}(z) = \frac{\delta_z}{\norm{z}_2} = \frac{1}{\frac{\norm{z}_2}{\delta_z}} \leq \frac{1}{\norm{z^{(p)}}_2} = \tau^p_{\pm}(z^{(p)}).
\end{equation*}

Next, we prove that (\ref{optproblemgravity}) is equivalent to (\ref{optproblemgravitypolyhedron}).
First we motivate the choices of $\gamma_1^p,\dots,\gamma_t^p$. The ``distance constraints'' of (\ref{optproblemgravity}) define two half spaces for each $j \in [t]$. Formally, we would have to replace $F_j$ by $\lin(F_j)$, but this becomes irrelevant due to only considering vectors in the normal cone. With $z^*_j \in B^p_1(0)$ being an optimal solution and $\gamma_j^p := v_j^T z^*_j < 0$ its objective value of (\ref{optproblemgravitydistance}), we obtain the equivalence
\begin{equation} \label{pnormdistancehalfspaces}
dist(z,\lin(F_j))_p \geq 1 \quad \Longleftrightarrow \quad z \in H^{\leq}_{(v_j,\gamma_j^p)} \cup H^{\geq}_{(v_j,- \gamma_j^p)}.
\end{equation}

Furthermore, for all $j \in [t]$ we know that $F_j = \lin(F_j) \cap N_{\mathcal{P}^{\pm}}(w(C))$, where $\lin(F_j) = \{v_j\}^{\perp}$ is a $(d \cdot k - 1)$-dimensional linear subspace, i.e. a hyperplane through the origin with normal vector $v_j$. In particular, we can write any $z \in \R^{d \cdot k}$ as linear combination of an orthogonal basis $\{ v_j,f^j_1,\dots,f^j_{d \cdot k -1 } \}$ of $\R^{d \cdot k}$ with $f^j_l \in \lin(F_j)$ for all $l \in [d \cdot k - 1]$, i.e.
\begin{equation*}
z = \nu^j_0 v_j + \sum_{l = 1}^{ d \cdot k -1} \nu^j_l f^j_l \quad \text{ with } \nu^j_0,\dots,\nu^j_{d \cdot k -1} \in \R.
\end{equation*}
For $z \in N_{\mathcal{P}^{\pm}}(w(C))$, we observe that $\nu^j_0 \leq 0$, since $v_j$ points along an edge direction, and thus away from the normal cone. For $z \in \interior(N_{\mathcal{P}^{\pm}}(w(C)))$, we observe that $\nu_0^j < 0$. Together with (\ref{pnormdistancehalfspaces}), we obtain 
\begin{equation*}
N_{\mathcal{P}^{\pm}}(w(C)) \cap \{ z \in \R^{d \cdot k } \ | \ dist(z,F_j)_p \geq 1 \} = N_{\mathcal{P}^{\pm}}(w(C)) \cap H^{\leq}_{(v_j, \gamma_j^p)}
\end{equation*}
for all $j \in [t]$. Rewriting the normal cone as
\begin{equation*}
N_{\mathcal{P}^{\pm}}(w(C)) = \bigcap_{j = 1}^{t} H^{\leq}_{(v_j,0)},
\end{equation*} 
and recalling that $\gamma_j^p < 0$, we see that the halfspace $H^{\leq}_{(v_j , \gamma_j^p)}$ is strictly contained in $H^{\leq}_{(v_j,0)}$ for all $j \in [t]$. Thus, 
\begin{equation*}
 \bigcap_{j = 1 }^{t} H^{\leq}_{(v_j, \gamma_j^p)} \subsetneq \bigcap_{j = 1}^{t} H^{\leq}_{(v_j,0)} = N_{\mathcal{P}^{\pm}}(w(C)).
\end{equation*}
Therefore, the constraint $z \in N_{\mathcal{P}^{\pm}}(w(C))$ is redundant and (\ref{optproblemgravity}) is equivalent to (\ref{optproblemgravitypolyhedron}). \hfill \qed
\endproof

We choose the $p$-norm distance to the facets of the cone as a constraint in (\ref{optproblemgravity}), because lower dimensional faces blocking the ball do not guarantee a stable site vector. To see this, consider the following $3$-dimensional example for $p =2$ and suppose the edges of the cone are used to block the unit ball rather than the facets.

Let $z=(0,0,2)^T$, $a^1=(0,\sqrt{3},3)^T$, $a^2 = (0,-\sqrt{3},3)^T$, $a^3=(\sqrt{3},0,3)^T \in \R^3$ and consider the cone spanned by $\{a^1,a^2,a^3\}$, see Figure \ref{fig:counterexampleedges} for cross-sections orthogonal to the $x_1$-axis (left) and orthogonal to the $x_3$-axis (right). Note that $a^1,a^2$ and $a^3$ are linearly independent. We have $z = \frac{1}{3} \cdot a^1 + \frac{1}{3} \cdot a^2 + 0 \cdot a^3$, which implies $z \in \bd(\pos(\{a^1,a^2,a^3\})$.
One can easily verify that $dist(z,\pos(\{a^j\}))_2 = \norm{z - \frac{1}{2} a^j}_2 = 1$ for all $j \in [3]$, i.e. $z$ is feasible and all constraints are tight, if we consider the edges instead of the facets of the cone in (\ref{optproblemgravity}). Note that we cannot decrease the (squared) norm of $z$ any further. However, this optimal solution has stability $\tau^2_{\star}(z) = 0$, i.e. this site vector cannot be perturbed in every direction without leaving the cone.

\begin{figure}
\begin{center}
\scalebox{1.15}{
\begin{tikzpicture}
	\coordinate (O) at (0,0);
	\coordinate (Z) at (0,2);
	\coordinate (A1) at (1.732050808,3);
	\coordinate (A2) at (-1.732050808,3);
	
	\fill[blue!15] (O) -- (A1) -- (A2) -- cycle;
	\fill[red!15] (Z) circle [radius=1cm];
	\draw[blue] (O) -- (A1);
	\draw[blue] (O) -- (A2);
	\draw[red] (Z) circle [radius=1cm];
	
	\fill[blue] (O) circle (0.05) node[below]{$0$};
	\fill[blue] ($0.5*(A1)$) circle (0.05) node[below right]{$\frac{1}{2} a^1$};
	\fill[blue] ($0.5*(A2)$) circle (0.05) node[below left] {$\frac{1}{2} a^2$};
	\fill[red] (Z) circle (0.05) node [above] {$z$};

	\coordinate (B) at (2,-0.5);
	\coordinate (B1) at (2.5,-0.5);
	\coordinate (B2) at (2,0);
	
	\draw[->, thick] (B) -- (B1) node[right] {$x_2$};
	\draw[->, thick] (B) -- (B2) node[above] {$x_3$};
\end{tikzpicture}
}
\hspace*{0.75cm}
\scalebox{1.15}{
\begin{tikzpicture}
	\coordinate (Z) at (0,0);
	\coordinate (A1) at (1.732050808,0);
	\coordinate (A2) at (-1.732050808,0);
	\coordinate (A3) at (0,1.732050808);
	
	\fill[blue!15] (A1) -- (A2) -- (A3) -- cycle;
	\draw[blue] (A1) -- (A2) -- (A3) -- cycle;
	\fill[red!15] (Z) circle [radius=1cm];
	\draw[red] (Z) circle [radius=1cm];
	\draw[blue, dashed] (A1) -- (A2);
	
	\fill[blue] (A1) circle (0.05) node[above right]{$\frac{2}{3} a^1$};
	\fill[blue] (A2) circle (0.05) node[above left] {$\frac{2}{3} a^2$};
	\fill[blue] (A3) circle (0.05) node[above] {$\frac{2}{3} a^3$};
	\fill[red] (Z) circle (0.05) node[below]{$z$};
	
	\coordinate (B) at (2.5,-1);
	\coordinate (B1) at (3,-1);
	\coordinate (B2) at (2.5,-0.5);
	
	\draw[->, thick] (B) -- (B1) node[right] {$x_2$};
	\draw[->, thick] (B) -- (B2) node[above] {$x_1$};
	
\end{tikzpicture}
}
\caption{Center $z$ of a stuck $2$-norm unit ball blocked by edges with stability $\tau^2_{\star}(z) = 0$.}
\label{fig:counterexampleedges}
\end{center}
\end{figure}
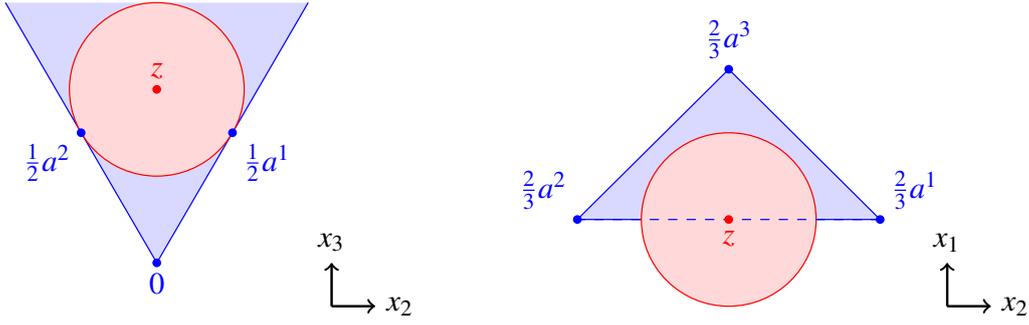

Similar examples can be constructed for any $p \in [1;\infty]$ and any $l$-dimensional face of $N_{\mathcal{P}^{\pm}}(w(C))$ that is not a facet. The case $l=0$ (the origin blocking the ball) does not provide any information, because every center has $p$-norm distance one to the origin. Let us consider the difference between using $l$-dimensional faces and using the facets $F_1,\dots,F_t$ of the normal cone to block the unit ball:

For every $j \in [t]$, the linear subspace $\lin(F_j)$ is a hyperplane that separates the underlying space into two halfspaces. In contrast, the linear subspace $\lin(F)$ of an $l$-dimensional face $F$ ($1 \leq l < d \cdot k - 1$) is not a hyperplane. This implies that a representation of the ``distance constraints'' $dist(z,\lin(F))_p \geq 1$ is not possible when using lower-dimensional faces. Therefore, the arguments of the proof of Theorem \ref{solveunitballincone} do not apply to lower dimensional faces. Informally, the center of the ball can ``orbit'' $\lin(F)$ at $p$-norm distance $1$ and eventually ``hit'' the boundary of the cone.


\proof{Proof of Theorem \ref{pqnormegal}.}
Let $q \in [1;\infty]$ with $q \not= p$, else the claim is trivial, because $c_{p,p} = 1$. By assumption $\B^p_1(z^{(p)}) \subseteq N_{\mathcal{P}^{\pm}}(w(C))$. Thus $\B^p_{\lambda} (\lambda z^{(p)} ) \subseteq N_{\mathcal{P}^{\pm}}(w(C))$ for all $\lambda > 0$. It suffices to show that $\lambda = c_{p,q}$ satisfies $\B^q_1(0) \subseteq \B^p_{c_{p,q}}(0)$.
For $y \in \B^q_1(0)$ we obtain
\begin{equation*}
1 \geq \norm{y}_q \geq \frac{1}{c_{p,q}} \norm{y}_p \quad \Longrightarrow \quad \norm{y}_p \leq c_{p,q}.
\end{equation*}
Hence, choosing $\lambda = c_{p,q} > 0$ yields $\B^q_1(c_{p,q} z^{(p)}) \subseteq \B^p_{c_{p,q}}(c_{p,q} z^{(p)} ) \subseteq N_{\mathcal{P}^{\pm}}(w(C))$.

Now let $z^{(q)}$ be the optimal solution of (\ref{optproblemgravitypolyhedron}) w.r.t. the $q$-norm, i.e. with constraints $dist(z,F_j)_q \geq 1$ for all $j \in [t]$. In the following, whenever we refer to (\ref{optproblemgravitypolyhedron}), we refer to the optimization problem w.r.t. the $p$-norm. Let $\gamma^q_1,\dots,\gamma^q_t$ denote the optimal objective values of (\ref{optproblemgravitydistance}) with $\B^q_1(0)$ as feasible region, and let $\gamma^p_1,\dots,\gamma^p_t$ denote the same for $\B^p_1(0)$. We distinguish between the two cases $p < q$ and $p > q$.

For $p < q$, we know $\B_1^p(0) \subseteq \B^q_1(0)$ and, therefore, $z^{(q)}$ satisfies $v_j^T z^{(q)} \leq  \gamma^q_j \leq  \gamma^p_j$ for all $j \in [t]$. Hence, $z^{(q)}$ is feasible for (\ref{optproblemgravitypolyhedron}) and $\B^p_1(z^{(q)}) \subseteq N_{\mathcal{P}^{\pm}}(w(C))$. The objective value of $z' = c_{p,q} z^{(p)}$ satisfies
\begin{equation} \label{pqnormegalcase1}
\norm{z'}_2^2 =  c_{p,q}^2 \norm{z^{(p)}}^2_2 \leq c_{p,q}^2 \norm{z^{(q)}}^2_2,
\end{equation}
because $z^{(p)}$ is an optimal solution of (\ref{optproblemgravitypolyhedron}).

For $p > q$, consider $\bar{z} = c_{q,p} z^{(q)}$, which satisfies $\B^p_1(\bar{z}) \subseteq N_{\mathcal{P}^{\pm}}(w(C))$. Note that $z^{(q)}$ is not feasible for (\ref{optproblemgravitypolyhedron}), but $\bar{z}$ is. Since $p > q$, we have $c_{p,q} = 1$, i.e. $z' = z^{(p)}$ and hence, by optimality of $z^{(p)}$,
\begin{equation} \label{pqnormegalcase2}
\norm{z'}^2_2 = \norm{z^{(p)}}^2_2 \leq \norm{ \bar{z}}^2_2 = c_{q,p}^2 \norm{z^{(q)}}^2_2.
\end{equation}
In total, since either $c_{p,q}$ or $c_{p,q}$ equals 1, (\ref{pqnormegalcase1}) and (\ref{pqnormegalcase2}) yield 
\begin{equation*}
\norm{z'}^2_2 \leq \max\{c_{p,q},c_{q,p}\}^2 \norm{z^{(q)}}^2_2.
\end{equation*}This proves the claim.\hfill \qed
\endproof

\section{Concluding Remarks}\label{sec-outlook}

We would like to stress that all computations of the volume of the normal cone, the computations of edges of the normal cone, as well as the optimal solution of (\ref{optproblemgravitypolyhedron}) are challenging due to the possibly exponential number of edges incident to a vertex of the polytope. However, this significant computational effort is rewarded with deep insights into the structure and behavior of site vectors. The key takeaway is that clusterings of large volume, i.e. vertex clustering with large normal cones, are both good and likely to be found by clustering algorithms. The next natural steps include the design of approximation algorithms that trade small errors in the construction of normal cones for a significant reduction in computation times.

%
%
%




\bibliographystyle{plainnat}
\bibliography{GoodClusteringsHaveLargeVolume} 

\end{document}